\documentclass[reqno]{amsart}
 \newtheorem{thm}{Theorem}[section]
 
 \newtheorem{lem}[thm]{Lemma}
 \newtheorem{prop}[thm]{Proposition}
 \theoremstyle{definition}
 \newtheorem{defn}[thm]{Definition}
 \theoremstyle{remark}
 \newtheorem{rem}[thm]{Remark}
 
 \numberwithin{equation}{section}

\usepackage{amsmath}
\usepackage{amssymb}
\usepackage{amsfonts}
\usepackage[sans]{dsfont}
\usepackage{appendix}
\pdfoutput=1
\usepackage[pdftex]{graphicx,color}

\begin{document}

%
%
%
%
%
%
%
%

\newcommand{\Z}{{\mathbb Z}}
\newcommand{\Q}{{\mathbb Q}}
\newcommand{\R}{{\mathbb R}}
\newcommand{\C}{{\mathbb C}}
\newcommand{\N}{{\mathbb N}}
\newcommand{\ttt}{|\hspace{-0.25mm}|\hspace{-0.25mm}|}
\renewcommand{\Re}{\mathop{\rm Re}\nolimits}
\renewcommand{\Im}{\mathop{\rm Im}\nolimits}

\def\sn{{\bf S}^{n-1}}
\def\ts{\tilde{\sigma}}
\def\ss{{\mathcal S}}
\def\aa{{\mathcal A}}
\def\cc{{\mathcal C}}
\def\tpp{\widetilde{{\mathcal P}}}
\def\iu{\underline{i}}
\def\lc{{\mathcal L}}
\def\kc{{\mathcal K}}
\def\pxi{\phi + (\xi + i u)\psi_p}
\def\pp{\mathcal P}
\def\rr{\mathcal R}
\def\hU{\widehat{U}}
\def\mt{\Lambda}
\def\hLa{\widehat{\mt}}
\def\ep{\epsilon}
\def\tPhi{\widetilde{\Phi}}
\def\hR{\widehat{R}}
\def\oV{\overline{V}}
\def\uu{\mathcal U}
\def\tz{\tilde{z}}
\def\hz{\hat{z}}
\def\hd{\hat{\delta}}
\def\ty{\tilde{y}}
\def\hs{\hat{s}}
\def\hc{\hat{\cc}}
\def\hC{\widehat{C}}
\def\hh{\mathcal H}
\def\tf{\tilde{f}}
\def\of{\overline{f}}
\def\trr{\tilde{r}}
\def\tr{\tilde{r}}
\def\tts{\tilde{\sigma}}
\def\tVl{\widetilde{V}^{(\ell)}}
\def\tVj{\widetilde{V}^{(j)}}
\def\tVo{\widetilde{V}^{(1)}}
\def\tVj{\widetilde{V}^{(j)}}
\def\tPsi{\tilde{\Psi}}
 \def\tp{\tilde{p}}
 \def\tVjo{\widetilde{V}^{(j_0)}}
\def\tvj{\tilde{v}^{(j)}}
\def\tVjj{\widetilde{V}^{(j+1)}}
\def\tvl{\tilde{v}^{(\ell)}}
\def\tVt{\widetilde{V}^{(2)}}
\def\Lo{\; \stackrel{\circ}{L}}
\def\tg{\tilde{g}}
\def\hq_n{\hat{q}_n}

\def\ii{{\imath }}
\def\jj{{\jmath }}
\vspace*{0,8cm}
\def\saa{\Sigma_A^+}
\def\sAA{\Sigma_{\aa}^+}
\def\Lip{\mbox{\rm Lip}}
\def\clip{C^{\mbox{\footnotesize \rm Lip}}}
\def\lip{\mbox{{\footnotesize\rm Lip}}}
\def\Vol{\mbox{\rm Vol}}

\def\gl{\gamma_\ell}
\def\id{\mbox{\rm id}}
\def\piU{\pi_U}
\def\bs{\bigskip}
\def\ms{\medskip}
\def\Int{\mbox{\rm Int}}
\def\diam{\mbox{\rm diam}}
\def\di{\displaystyle}
\def\dist{\mbox{\rm dist}}
\def\ff{{\cal F}}
\def\i{{\bf i}}
\def\pr{\mbox{\rm pr}}
\def\co{\; \stackrel{\circ}{C}}
\def\la{\langle}
\def\ra{\rangle}
\def\supp{\mbox{\rm supp}}
\def\Arg{\mbox{\rm Arg}}
\def\Int{\mbox{\rm Int}}
\def\II{{\mathcal I}}
\def\e{\emptyset}
\def\endofproof{{\rule{6pt}{6pt}}}
\def\con{\mbox{\rm const }}
\def\Box{\spadesuit}
\def\be{\begin{equation}}
\def\ee{\end{equation}}
\def\beqn{\begin{eqnarray*}}
\def\eeqn{\end{eqnarray*}}
\def\MM{{\mathcal M}}
\def\tmu {\tilde{\mu}}
\def\Pr{\mbox{\rm Pr}}
\def\Prf{\mbox{\footnotesize\rm Pr}}
\def\htau{\hat{\tau}}
\def\btau{\overline{\tau}}
\def\hr{\hat{r}}
\def\tF{\widetilde{F}}
\def\tG{\widetilde{G}}
\def\trho{\tilde{\rho}}

\title[Sharp large deviations] {Sharp large deviations for hyperbolic flows}

\author[V. Petkov]{Vesselin Petkov}
\address{Universit\'e de Bordeaux, Institut de Math\'ematiques de Bordeaux, 351,
Cours de la Lib\'eration, 
33405  Talence, France}
\email{petkov@math.u-bordeaux.fr}
\author[L. Stoyanov]{Luchezar Stoyanov}
\address{University of Western Australia, School of Mathematics 
and Statistics, Perth, WA 6009, Australia}
\email{stoyanov@maths.uwa.edu.au}



\def\ts{\tilde{\sigma}}
\def\ss{{\mathcal S}}
\def\aa{{\mathcal A}}
\def\R{{\mathbb R}}
\def\C{{\mathbb C}}
\def\iu{\underline{i}}
\vspace*{0,8cm}
\def\saa{\Sigma_A^+}
\def\sAA{\Sigma_{\aa}^+}
\def\lc{{\mathcal L}}
\def\pxi{\phi + (\xi + i u)\psi_p}
\def\sa{\Sigma_A}
\def\ssa{\Sigma_A^+}
\def\ssn{\sum_{\sigma^n x = x}}
\def\lc{{\mathcal L}}
\def\lo{{\mathcal O}}
\def\oo{{\mathcal O}}
\def\mt{\Lambda}
\def\ep{\epsilon}
\def\ms{\medskip}
\def\bs{\bigskip}
\def\diam{\mbox{\rm diam}}
\def\rr{\mathcal R}
\def\pp{\mathcal P}
\def\hU{\widehat{U}}
\def\hz{\hat{z}}
\def\tz{\tilde{z}}
\def\ty{\tilde{y}}
\def\tx{\tilde{x}}
\def\i{{\bf i}}
\def\pp{{\mathcal P}}
\def\ff{{\mathcal F}_{\theta}}
\def\tG{\tilde{G}}
\def\hw{\hat{w}}
\def\ep{\epsilon}
\def\rt{R^{\tau}}
\def\stt{\sigma_{\tau}^t}
\def\m{{\sf m}}
\def\lc{{\mathcal L}}
\def\vt{\varphi^{\tau}}
\def\tF{\tilde{F}}
\def\pc{{\mathcal P}}
\def\tq{\tilde{q}}
\def\hU{\widehat{U}}
\def\hR{\widehat{R}}
\def\hrt{\widehat{R}^\tau}

\begin{abstract}
 For hyperbolic flows $\varphi_t$ we examine  the Gibbs measure of points $w$ for which
$$\int_0^T G(\varphi_t w) dt - a T \in (- e^{-\ep n}, e^{- \ep n})$$
as $n \to \infty$ and $T \geq n$, provided  $\ep > 0$  is sufficiently small. 
This is similar to local central limit theorems. The fact that the interval $(- e^{-\ep n}, e^{- \ep n})$ is exponentially shrinking as $n \to \infty$ leads to several difficulties.  
Under some geometric assumptions we establish a sharp large deviation result with leading term $C(a) \ep_n e^{\gamma(a) T}$ and rate function $\gamma(a) \leq 0.$ 
The proof is based on the spectral estimates for the iterations of the Ruelle operators with two complex parameters and on a new Tauberian theorem for sequence of 
functions $g_n(t)$ having an asymptotic as $ n \to \infty$ and  $t \geq n.$

\end{abstract}

\maketitle

\section{Introduction}
\renewcommand{\theequation}{\arabic{section}.\arabic{equation}}
\setcounter{equation}{0}

Let  $\varphi_t: M \longrightarrow M$ be a $C^2$ weak mixing Axiom A flow on a compact Riemannian  manifold $M$, and let $\mt$ be a basic set 
for $\varphi_t$. The restriction of the flow on 
$\Lambda$ is a hyperbolic flow \cite{PP}. For any $x \in M$ let $W_{\epsilon}^s(x), W_{\epsilon}^u(x)$ be the local stable and
unstable manifolds through $x$,  respectively (see \cite{B2}, \cite{KH},  \cite{PP}). 
It follows from the hyperbolicity of $\mt$  that if  $\epsilon_0 > 0$ is sufficiently small,
there exists $\ep_1 > 0$ such that if $x,y\in \mt$ and $d (x,y) < \ep_1$, 
then $W^s_{\ep_0}(x)$ and $\varphi_{[-\ep_0,\ep_0]}(W^u_{\ep_0}(y))$ intersect at exactly 
one point $[x,y ] \in \mt$  (cf. \cite{KH}). That is, there exists a unique 
$t\in [-\ep_0, \ep_0]$ such that $\varphi_t([x,y]) \in W^u_{\ep_0}(y)$. Setting $\Delta(x,y) = t$, 
defines the so called {\it temporal distance function}. Here and throughout the whole paper we denote by $d(\cdot , \cdot)$ 
the {\it distance} on $M$ determined by the Riemannian metric.

Let $\rr = \{ R_i\}_{i=1}^k$ be a fixed  {\it (pseudo) Markov family}  of {\it pseudo-rectangles} 
$R_i = [U_i  , S_i ] =  \{ [x,y] : x\in U_i, y\in S_i\}$ (see Section 2). Set $R = \cup_{i=1}^{k} R_i,\: U = \cup_{i=1}^k U_i$. Consider the
{\it Poincar\'e map} $\pp: R \longrightarrow R$, defined by  $\pp(x) = \varphi_{\tau(x)}(x) \in R$, where
$\tau(x) > 0$ is the smallest positive time with $\varphi_{\tau(x)}(x) \in R$ ({\it first return time}  function). 
The  {\it shift map}  $\sigma : R   \longrightarrow U$ is given by
$\sigma  = \pi_U \circ \pp$, where $\pi_U : R \longrightarrow U$ is the {\it projection} along stable leaves.

Define a $(k \times k)$ matrix $A = \{A(i, j)\}_{i, j = 1}^k$ by
$$A(i, j) = \begin{cases} 1 \:\:{\rm if}\: \pp({\rm Int}\:R_i)\cap {\rm Int}\: R_j \neq \emptyset,\\
0 \:\:{\rm otherwise}. \end{cases}$$
Following \cite{B2}, it is possible to construct a Markov family ${\mathcal R}$ so that $A$ is irreducible and aperiodic.

Consider the suspension space 
$$R^{\tau} = \{(x, t) \in R \times \R:\: 0 \leq t \leq \tau(x)\}/ \sim,$$
where by $\sim$ we identify the points $(x, \tau(x))$ and  $(\pp x, 0).$ 
The suspension  flow  on $\rt$ is defined by $\vt_t(x, s) = (x, s + t)$ 
taking into account the identification $\sim.$
 For a H\"older continuous function $f$ on $R$, the {\it topological pressure} $\Pr_{\pp}(f)$ with respect to $\pp$ is defined by
$$\Pr_{\pp}(f) = \sup_{m \in {\mathcal M}_{\pp}} \big\{ h(\pp, m) + \int f d m\big \},$$
where ${\mathcal M}_{\pp}$ denotes the space of all $\pp$-invariant Borel probability measures and $h(\pp, m)$ is 
the  entropy of $\pp$ with respect to $m$.  We say that $u$ and $v$ are {\it cohomologous} and we denote this by 
$u \sim v$ if there exists a continuous function $w$ such that $u = v + w \circ \pp- w.$ The flow $\varphi_t$ on $\mt$ is 
naturally related to the suspension flow  $\vt_t$ on 
$R^{\tau}$. There exists a natural semi-conjugacy projection $\pi(x, t): R^{\tau} \longrightarrow \Lambda$
which is one-to-one on a residual set (see \cite{B2}) such that
$\pi(x, t) \circ \vt_s = \vt_s \circ \pi(x, t).$ 
 For $z \in R$ set
$$\tau^n(z): = \tau(z) + \tau(\pp(z)) +...+ \tau(\pp^{n-1}(z)).$$
Notice that since $\tau(x)$ is constant along stable leaves for $x = \pi_U (z)$ we have
$$\tau^n(z) = \tau^n (x) = \tau(x) + \tau(\sigma(x)) + ...+ \tau(\sigma^n(x)).$$

Denote by $\widehat{U}$ (or $\hR$) the set  of those $x\in U$ (resp. $x \in R$)
such that  $\pp^m(x)$ does not belong to the bounder of any rectangle $R_i$  for all $m \in \Z$. 
In a similar way define $\hrt$.
It is well-known (see \cite{B1}) that $\hU$ (resp. $\hR$)  is a residual subset 
of $U$ (resp. $R$) and has full measure with respect to any Gibbs measure on $U$ (resp. $R$).
Clearly in general $\tau$ is not continuous on $U$, however $\tau$ is {\it essentially H\"older}  on $U$, i.e.
there exist constants $L>0$ and $\alpha >0$ such that $|\tau(x) - \tau(y)| \leq L\, (d(x,y))^\alpha$
whenever $x,y \in U_i$ and $\sigma(x), \sigma(y)\in U_j$ for some $i, j$. 
The same applies to $\sigma : U \longrightarrow U$ and to $\pp : R \longrightarrow R$.   Throughout we will mainly 
work with the restrictions of $\tau$ and $\sigma$ to $\hU$ and also with the restrictions of $\tau$ and $\pp$ to $\hR$.

Consider the space $C^{\alpha}(\hrt)$ of all $\alpha$-H\"older functions on $\hrt$ with norm $\|w\|_{\alpha} = |w|_{\alpha} + \|w\|_{\infty}$. 
We should stress that throughout the whole paper the H\"older norms $|w|_\alpha$ for functions on $\hU$, $\hR$ or $\hrt$ are always 
determined with respect to  distance $d(\cdot , \cdot)$ on $M$ determined by the Riemannian metric. 

For $F \in C^{\alpha}(\hrt)$, define the function $f : R \longrightarrow \R$ by
$$f(z) = \int_0^{\tau(z)} F(z,t) \; dt , \quad z\in R .$$
 Here and in the following we use the notation $F(z, t) = F(\sigma_t^{\tau}(z, 0)).$  Given a function $G \in C(\hrt)$, we define 
$$G^T (w)= \int_0^T G(\vt_t(w) )\; dt, \quad w \in \rt .$$

 Throughout the paper we assume that $F, G\in C^{\alpha}(\hrt)$.
Let  $a = \int_{\rt} G dm_{F + t G}$, where $ m_{F + t G}$ is the equilibrium state of $F + t G$ for some $t \in \R$. 
More precisely, for a function $A$ on $\hrt$ the equilibrium state $m_A$ of $A$ is a $\sigma^{\tau}_t$-invariant probability measure on $\rt$ such that
$$\Pr_{\varphi_t^{\tau}}(A) = \sup_{m \in {\mathcal M}_{\tau}}\{h(\varphi_1^{\tau}, m) + \int A(w) \; dm(w)\} ,$$
$h(\varphi_1^{\tau}, m)$ being the {\it entropy} of $\varphi_1^{\tau}$ with respect to $m$ and ${\mathcal M}_{\tau}$ the space of 
all $\varphi_t^{\tau}$ invariant 
Borel probability mesures on $\rt.$ The supremum above is given by a measure $m_A$ called equilibrium state of $A$.
Moreover, if $G$ is not cohomologous to a constant, we have
$$0 < \sigma_{m}^2(G) = \lim_{T \to \infty} \frac{1}{T} \left( \int_0^T (G \circ \sigma_t^{\tau}) dt - T \int G dm \right)^2 < \infty$$
and $\frac{d^2\Pr(F + t G)}{dt^2} = \sigma_{m_{F + t G}}^2(G),$ $m_H$ being the {\it equilibrium state} of $H$.

Introduce the rate function 
$$\gamma(a): = \inf_{t \in \R} \{ \Pr(F + t G) - \Pr(F) - ta\} = \Pr(F + \xi(a) G) - \Pr(F) - \xi(a) a ,$$
where $ \xi(a)$ is the unique real number such that
$$\frac{d \Pr(F + t G)}{dt} \bigg\vert_{t = \xi(a)} = \int G\, dm_{F + \xi(a) G} = a$$
and $\Pr(A) = \Pr_{\sigma^{\tau}_t}(A)$ is the {\it pressure} of $A$ with respect to the flow $\sigma^{\tau}_t$ on $\rt$.

 For simplicity of the notations we will write $\Pr(A)$ instead of $\Pr_{\vt_t}(A).$ Let 
$$\beta(t) = \Pr(F + tG) - \Pr(F).$$
Clearly,
$$\gamma'(a) = \beta'(\xi(a)) \xi'(a) - \xi'(a) a - \xi(a) = - \xi(a),$$
and 
$$\xi'(a) = \frac{1}{\beta''(\xi(a))} = \frac{1}{\sigma^2_{m_{F + \xi(a) G}}(G)}.$$
Consequently, $\gamma'(a) = 0$ if and only if $\xi(a) = 0$ which is equivalent to $\int G \, dm_F = a.$ Thus $\gamma(a)$ is a {\it non-positive} 
concave function and  $\gamma(a) = 0$ if only if $\int G \, dm_F = a.$\\

In this paper we continue the analysis of sharp large deviations in \cite{PeS1}, \cite{PeS3}. Our purpose is to improve the results in \cite{PeS3} and to study for a fixed $q \geq 0$ the asymptotic of
$$ m_F\Bigl\{w \in \rt: \forall T \geq n -q\:\: \mbox{\rm we have} \:\:  \int_0^T G(\vt_t(w)) dt - a T \in \Bigl(-e^{-\ep n}, e^{-\ep n}\Bigr) \Bigr\}$$
as $ n \to \infty.$
Here $0 < \ep \leq \mu_0 /8$ is a small constant, where $\mu_0 >0$  defined in Section 4 is related to the meromorphic continuation of the function $Z(s,\omega,a)$ across the line $\Re s = \gamma(a).$

Recall the subset $\hU$ of $\hR$ defined earlier. We will regard $\hU$ as a subset of $\hrt$ using the identification $x \longleftrightarrow (x,0)$ for $x \in R$. 
Now we introduce two definitions of independence.

\begin{defn} Two functions $f_1, f_2 \in C^{\alpha}(\hU)$ are called $\sigma$-independent if whenever there are constants $t_1, t_2 \in \R$ such that
$t_1 f_1 + t_2 f_2$ is cohomologous to a function in $C(\hU : 2 \pi \Z)$, we have $t_1 = t_2 = 0.$
\end{defn}

For a function $G \in C^{\alpha}(\hrt)$ consider the skew product flow $S_t^G$ on ${\mathbb S}^1 \times \rt$ defined  by
$$S_t^G ( e^{2 \pi \i \alpha}, y) = \Bigl(e^{2 \pi \i (\alpha + G^t(y))}, \varphi_t^{\tau}(y)\Bigr).$$
\begin{defn} [\cite{La}] Let $G \in C^{\alpha}(\hrt)$. Then $G$ and $\varphi_t^{\tau}$ are {\it flow independent} if 
the following condition is satisfied. If $t_0, t_1 \in \R$ are constants such that the skew product flow $S_t^H$ with $H = t_0 + t_1 G$ is not topologically transitive, then $t_0 = t_1 = 0.$
\end{defn}
Following the result in \cite{La},  if $G$ and $\varphi_t^{\tau}$ are flow independent, then the flow $\varphi_t^{\tau}$ is topologically weak-mixing and the function $G$ is not cohomologous to a constant function. On the other hand, if $G$ and $\varphi^{\tau}_t$ are flow independent, then $g(x) = \int_0^{\tau(x)} G(x, t) dt, \: x \in U$ and $\tau(x)$ are $\sigma$-independent. \\

Introduce the set
$$\Gamma_G = \Big \{ \int G dm_{F + t G} :\: t \in \R\Big \}.$$
Our first result is the following
\begin{thm}
Assume that the Standing Assumptions stated in Section $2$ below are satisfied.    
Let   $G: \Lambda \longrightarrow (0,\infty)$ be a H\"older function function for which
there exists a Markov family $\rr = \{ R_i\}_{i=1}^k$ for the flow $\varphi_t$ on $\mt$ such that
$G$ is constant on the stable leaves of  all  "rectangular boxes"
$$B_i = \{ \varphi_t(x) : x\in R_i, 0 \leq t \leq \tau(x)\} ,$$ $i= 1, \ldots, k$.
Assume in addition that $G$ and  $\varphi_t^{\tau}$ are flow independent.  
Let $q \geq 0$ be fixed, let $0 < \mu_0 < 1$ be  the constant in Proposition $3$ and let $\ep_n = e^{-\ep n}, \: 0 <\ep \leq \mu_0/8$. Then for any compact set $J \Subset \Gamma_G$ and $0 < \eta \ll 1$ 
there exists $n_0(\eta) \in \N$ such that for $a \in J, \:n \geq n_0(\eta) + q$ and $T \geq n_0(\eta)- q$ we have
\begin{equation*}
\frac{\sqrt{2}\ep_n C(a) }{\sqrt{ \pi T \beta''(\xi(a))}}e^{\gamma(a) T}(1- \eta) \leq m_F\Bigl\{w \in \rt:\: \int_0^T G(\vt_t(w)) \; dt - a T\in \Bigl(-e^{-\ep n}, e^{-\ep n}\Bigr) \Bigr\}
\end{equation*} 
\begin{equation}\label{eq:1.1}
\leq \frac{\sqrt{2}\ep_n C(a)}{\sqrt{ \pi T \beta''(\xi(a))}}e^{\gamma(a)T}(1+ \eta),
\end{equation}
where $C(a) > 0$ is a constant defined in Section $6$.
\end{thm}

\medskip

The above theorem says that for $T \geq n - q$ and $n \to \infty$ we have
\begin{eqnarray} \label{eq:1.2}
m_F\Bigl\{w \in \rt:\: \int_0^T G(\varphi_t^{\tau}(w) )\; dt - aT \in \Bigl(-e^{-\ep n}, e^{-\ep n}\Bigr) \Bigr\} \nonumber \\
\sim \frac{\sqrt{2}\ep_n C(a)}{\sqrt{ \pi T \beta''(\xi(a))}}e^{\gamma(a)T}.
\end{eqnarray}

Notice that  the proof of Theorem $1.3$ works if we assume that $G$ is not cohomologous to a constant and $\tau(x) $ and $g(x) = \int_0^{\tau(x)}  G(x, t) dt$ are $\sigma-$ independent. As it was mentioned above, these properties are satisfied if $G$ and $\varphi_t^{\tau}$ are flow independent.

Theorem 1.3  is an improvement of a result in \cite{PeS3}, where the asymptotic of
$$\mu_f \Bigl\{x \in U: \: \int_0^{\tau^n(x)} G(\vt_t(x, 0)) dt - a\tau^n(x) \in \Bigl(-e^{-\ep n}, e^{-\ep n}\Bigr) \Bigr\},\: n \to \infty$$
has been investigated, $ \mu_f$ being the equilibrium state of $f(x)$. 
When we replace $\tau^n (x)$ by $T$, we have to study two limits: $n \to \infty$ and $T \to \infty$, and the condition $T \geq n - q$ is natural. It is easer to study the case 
when $n$ is fixed and we take $T \geq T(\eta, n)$ to arrange 
(\ref{eq:1.1}). However assuming $n \geq n_0(\eta)$, for fixed $0 <\eta \ll 1$ we could have $\lim_{n \to \infty}T(\eta, n) = \infty$
and we cannot arrange an uniformity with respect to $n$. Thus, our result is much more precise and to prove it we follow a strategy based on 
Tauberian theorems with two parameters $n, T$ 
examining the asymptotic  of a sequence of functions $g_n(T).$ We discuss briefly this approach below.

\medskip
 
Our second result concerns the function
$$\zeta(T; a) = m_F\Bigl\{ w \in \rt: \int_0^T G(\vt_t(w)) \;  dt - a T\in \Bigl(-e^{-\ep T}, -e^{-\ep T}\Bigr) \Bigr\}.$$
Applying Theorem 1, we prove the following

\begin{thm} Under the assumptions of Theorem $1.3$, for $a \in J$ and any $0 < \eta \ll 1$ there exists $n_0(\eta) \in \N$ such that for $T \geq n_0(\eta) + 1$ we have
\begin{equation} \label{eq:1.3}
\frac{\sqrt{2}e^{-\ep} e^{-\ep T} C(a)}{\sqrt{ \pi T \beta''(\xi(a))}}e^{\gamma(a)T}(1-\eta) \leq \zeta(T; a) \leq \frac{\sqrt{2}e^{\ep} e^{-\ep T} 
C(a)}{\sqrt{\pi T \beta''(\xi(a))}}e^{\gamma(a)T}(1+ \eta).
\end{equation}
\end{thm}
 It is possible to obtain a slightly better result assuming one can generalise Theorem 1.3 for sequences $n_k \to \infty$ instead of a sequence of integers $n \to \infty$, however we 
 are not going to discuss such generalisations. 
 
 \medskip
 
 The results of the type discussed above are known as local central limit theorems (LCLT) (see \cite{DN} for recent results and references). In particular,
 (LCLT) in a very general setting are studied in \cite{DN} and asymptotics of the form
 $$m_F \Bigl\{ w \in \rt: \int_0^T G(\vt_t(w)) dt - T \int G dm_F - c \sqrt{T} \in I\Bigr\} \sim \frac{{\bf g}(c)}{\sqrt{T}} {\rm Leb} (I) ,\: T \to \infty$$
 are proved when $I$ is a bounded interval in $\R$, ${\bf g}(c)$ is a Gaussian density and  ${\rm Leb}(I)$ is the Lebesgue measure of $I$.
 The case considered in the present paper, where we deal with exponentially shrinking intervals $I_n \subset \R$ and we want to have an asymptotic as 
 $n \to \infty$ and $T \to \infty$, is more difficult. 
 Large deviations for Anosov flows have been examined by Waddington \cite {W}, where for the measure
 $$m_F\Bigl\{w \in \rt:\: \int_0^T G(\sigma^\tau_t(w)) \; dt - aT \in [c , d]\Bigr\} $$
it was  obtained an asymptotic similar to (\ref{eq:1.2}) with leading term having the form 
$$ \int_c^d e^{-\xi(a) t} dt\frac{C(a)}{\sqrt{2\pi T \beta''(\xi(a))}}e^{\gamma(a)T}.$$
 In \cite{W} there are several points presented without proofs. 
In the exposition in \cite{W} the case when $G$ is constant along stable leaves is treated, while in the general case  no argument is provided. This gap is essential since for large deviations,
applying a reduction based on Proposition 2.2 (see Section 2), new terms appear in the analysis of the Laplace transform when we work  with the 
iterations of Ruelle operators.  A second gap is related to Proposition 6(ii) in \cite{W} which is also presented without proof. This Proposition concerns a Tauberian theorem for a nonnegative function 
which is not monotonic. In general, without  a slowly decreasing condition (see Section 10  in \cite{Kor}), or without some condition on the growth  of the derivative, it is not known if the result is true.

\medskip

 In the analysis of the large deviations in the case when a fixed interval $[c, d]$ is replaced by an exponentially shrinking interval $(- e^{-\ep n},e^{-\ep n}),$ several additional difficulties appear.  In two previous papers \cite{PeS1}, \cite{PeS3} some partial cases have been treated, but in these papers we  have considered only limits $n \to \infty$ (see also \cite{PoS1}, where the case of an interval $(-n^{-\kappa}, n^{-\kappa})$ with suitable $\kappa > 1$ has been studied). In this paper we deal with two limits $n \to \infty$, $T \geq n -q$. 
Our approach is based on spectral estimates for the iterations of a Ruelle operator 
$$\lc_{s, \omega, a} = \lc_{f  - s \tau + (\xi(a) + \i \omega)(g- a\tau) },\: s \in \C, \: \omega \in \R$$
with {\it two complex parameters} $s$ and $\i \omega$ which may have modulus going to $\infty.$ We exploit the estimates obtained in \cite{PeS2}, \cite{PeS3} (see Theorem 2.1) in order to obtain an analytic continuation of 
the Laplace transform $F_n(s)$ of the function $g_n(T)$ defined below for $|\Im s| \geq M$ and $|\omega| \geq \ep_0$. We establish the existence of an analytic continuation of $F_n(s)$ across the line $\Re s = \gamma(a)$  for $\gamma(a) -\mu_0 \leq \Re s,\: |\Im s| \geq M \gg 1.$ This continuation and the corresponding estimates play a crucial role in the new type Tauberian theorems concerning  double limits $n \to \infty$, $T \to \infty$. 
These Tauberian  theorems are of independent interest.

\medskip

For convenience of the reader we explain briefly the idea of the proof of Theorem 1.3. Let $\chi(t) \in C_0^{\infty}(\R: \R^+)$ be a nonnegative cut-off function and let
$$G^T(w) = \int_0^T G(\varphi_t (w)) dt, \: w \in \rt.$$
Set $\chi_n(t) = \chi\Bigl( \frac{t}{\ep_n}\Bigr).$We study the function
$$g_n(T) : =  \ep_n e^{-\gamma(a) T + T} \int_{\rt} \chi \Bigl(\frac{ G^T - a T}{\ep_n}\Bigr)(w) dm_F(w)$$
$$ = \frac{\ep_n^2}{2\pi} e^{-\gamma(a) T + T} \int_{\rt} \int_\R e^{\i \omega(G^T - a T) (w)} \hat{\chi} (\ep_n \omega) d\omega dm_F (w),$$
where $\hat{\chi}_n (\omega) = \ep_n\hat{\chi}(\ep_n \omega)$ is the Fourier transform of $\chi_n(t).$ We extend this function as 0 for $T < 0$ and  examine the Laplace transform
$$F_n(s) =\frac{\ep_n^2} {2 \pi}\int_{\R} \Bigl[\int_0^{\infty}e^{-s T - \gamma(a) T + T}\Bigl( \int_{\rt} e^{\i \omega(G^T - a T) (w)} dm_F(w) \Bigr) dT \Bigr]\hat{\chi}(\ep_n \omega) d\omega.$$

Our purpose is to prove that for fixed $q \geq 0$, as $n \to \infty,\: T \geq n - q$, we have 
\begin{equation} \label{eq:1.4}
g_n(T) \sim \frac{\sqrt{2} C(a) \ep_n^2} {\sqrt{\pi \beta''(\xi(a))T}} e^T
\end{equation}
which yields the asymptotic (\ref{eq:1.2}). The factor $\ep_n$ in $g_n(T)$ is involved to have an independent of $n$ bound for the derivative $g_n'(T)$ (see Proposition 5.6 and Lemma 5.7 in Section 5).
Let
$$Z(s, \omega, a) = \int_0^{\infty} e^{- (s + \gamma(a) - 1)T} \Bigl( \int_{\rt} e^{\i \omega (G^T - a T) w} dm_F(w) \Bigr) dT.$$
 For fixed $a$ the function $Z(s, \omega, a)$ depends on two  complex parameters 
$s \in \C$ and $\i \omega$. Moreover, in $Z(s, \omega, a)$ we have no integration with respect to $\omega$. First we show that this function is analytic for $\Re s > \gamma(a).$  Second we prove that in  a small neighbourhood of $(\gamma(a), 0) \in \C^2$ 
this function has a pole $s(\omega, a)$ with residue $C(\omega, a) > 0$. To establish an analytic continuation across the line $\Re s = \gamma(a)$ for $|s- \gamma(a)| \geq \ep_0 > 0,$  first we reduce the integration on $R$,
and then by using the hypothesis that $G$ is constant along stable leaves, we reduce once more the integration on $U$ and write (see Section 3) 
$$Z(s, \omega, a) = \int_U B_2(s,\omega, a, u)\sum_{m = 0}^{\infty} \Bigl(\lc_{f- s \tau + \i \omega(g - a \tau)}^m B_1(s, \omega, a, .)\Bigr) (u)h(u)d\nu(u),$$
$\lc_{f - s \tau + \i \omega( g - a \tau)}$ being the Ruelle operator related to $f - s \tau + \i \omega (g - a \tau)$, where $f(x), g(x), \: x \in U,$ are determined by $F, G,$ respectevely, and the  measure $\nu(u)$ on $U$ is determined by $f(u)$. We exploit the estimates for the iterations $\lc^m_{f -s \tau + \i \omega(g - a \tau)}$ obtained in \cite{PeS3} (see Theorem 2.1 in Section 2) to obtain a meromorphic continuation of $Z(s, \omega, a)$ across $\Re s = \gamma(a)$.

\medskip

 The integration with respect to $\omega$ is not involved in the definition of $Z(s, \omega, a)$. Taking  the inyegration in a small interval $[-\ep_0, \ep_0]$, writing   
$\hat{\chi}(\ep_n \omega) = \hat{\chi}(0) + \hat{\chi}'(0) \ep_n \omega+ {\mathcal O}(\ep_n^2 \omega^2)$,  and repeating the calculus in \cite{KS}, \cite{W}, we get
$$\ep_n^2\int_{-\ep_0}^{\ep_0}    \frac{\hat{\chi}(\ep_n \omega)}{s  - 1 - s(\omega, a)} d\omega = \frac{C(a) \hat{\chi}(0)\ep_n^2} {\sqrt{2 \beta''(\xi(a))(s -1)}} + {\rm smoother}\:\:{\rm terms}.$$
The leading term becomes $\frac{A_n}{\sqrt{s-1}}$ with  $A_n = {\mathcal O} (\ep_n^2) = {\mathcal O} (e^{-2 \ep n})$ and this difficulty corresponds to  the type of Tauberain results proved in Section 5, where the remainder of the asymptotic has order $o(\ep_n^2).$\\

The integral with respect to $\omega$ over $\R \setminus [-\ep_0, \ep_0]$ yields analytic functions, however we need precise estimates on theirs growth as $|\Im s| \to \infty$ independent on $n$ in order to apply  Proposition 5.6. For this purpose, applying the results of Section 4, we are going to estimate  the integral
$$\ep_n^2 \int_{|\omega| \geq M \gg 1} (1 + |\Im s|^\nu + |\omega|^\nu) |\hat{\chi}(\ep_n \omega)| d \omega$$
with $0 < \nu < 1$, uniformly with respect to $n$. Here the presence of the factor $\ep_n^2$ is crucial and our choice of $\ep_n$ in the definition of $g_n(T)$ is once more very convenient.
To check the hypothesis of Proposition 5.6, we use in an essential way the analytic continuation of $Z(s, \omega, a)$ and its corresponding estimates.

\medskip

The plan of the paper is as follows. In Section 2 we introduce some definitions and our Standing assumptions.  A Sinai's lemma for the suspended flow is stated; it is proved
in the Appendix.  In Section 3 a representation of  the Laplace transform $Z(s, \omega, a)$ of the function $g_n(T)$ is obtained. Section 4 is devoted to the analysis of the 
meromorphic continuation of $Z(s, \omega, a)$ based, as mentioned above, on the results in 
\cite{PeS3}. In Section 5 two Tauberain theorems are proved for a sequence of nonnegative functions $g_n(T)$. The novelty here is that these functions have singularities 
$\frac{A_n}{\sqrt{s-1}}$ with  $0 < e^{-\mu n} \leq A_n \leq C_1$ and $0 < \mu \leq \mu_0/4$ and we have a double limit $n \to \infty,\: t \to \infty.$ Theorems 1.3 and 1.4 are proved in Section 6.

\medskip

By using Proposition 2.2, we can study also the general case when the function $G$ is not constant on stable leaves. A part of Proposition 4.2 in Section 4 concerning the analytic continuation 
for $|\Im s| \geq M$  can be established. However, some new  difficulties appear in the description of the singularities of $Z(s, \omega, a)$ given by Proposition 4.1 (ii). 
How to deal with this is an interesting open problem.

\section{Definitions and Standing Assumptions}
\renewcommand{\theequation}{\arabic{section}.\arabic{equation}}
\setcounter{equation}{0}

\subsection{Standing assumptions}

Throughout we use the notation and assumptions from the beginning of Section 1. In particular,
$\varphi_t: M \longrightarrow M$  is a $C^2$ weak mixing Axiom A flow  and $\mt$ is a basic set for $\varphi_t$.
As in \cite{PeS2} and \cite{PeS3}, we will work under the following rather general non-integrability condition
about the flow on $\mt$:

\bs

\noindent
{\sc (LNIC):}  {\it There exist $z_0\in \mt$,  $\ep_0 > 0$ and $\theta_0 > 0$ such that
for any  $\ep \in (0,\ep_0]$, any $\hz\in \mt \cap W^u_{\ep}(z_0)$  and any tangent vector 
$\eta \in E^u(\hz)$ to $\mt$ at $\hz$ with  $\|\eta\| = 1$ there exist  $\tz \in \mt \cap W^u_{\ep}(\hz)$, 
$\ty_1, \ty_2 \in \mt \cap W^s_{\ep}(\tz)$ with $\ty_1 \neq \ty_2$,
$\delta = \delta(\tz,\ty_1, \ty_2) > 0$ and $\ep'= \ep'(\tz,\ty_1,\ty_2)  \in (0,\ep]$ such that
$$|\Delta( \exp^u_{z}(v), \pi_{\ty_1}(z)) -  \Delta( \exp^u_{z}(v), \pi_{\ty_2}(z))| \geq \delta\,  \|v\| $$
for all $z\in W^u_{\ep'}(\tz)\cap\mt$  and  $v\in E^u(z; \ep')$ with  $\exp^u_z(v) \in \mt$ and
$\la \frac{v}{\|v\|} , \eta_z\ra \geq \theta_0$,   where $\eta_z$ is the parallel 
translate of $\eta$ along the geodesic in $W^u_{\ep_0}(z_0)$ from $\hz$ to $z$. }

\bs

Next, given $x \in \mt$, $T > 0$ and  $\delta\in (0,\ep]$ set
$$B^u_T (x,\delta) = \{ y\in W^u_{\ep}(x) : d(\varphi_t(x), \varphi_t(y)) \leq \delta \: \: , \:\:  0 \leq t \leq T \} .$$
We will say that $\varphi_t$ has a {\it regular distortion along unstable manifolds} over
the basic set $\mt$  if there exists a constant $\ep_0 > 0$ with the following properties:

\ms 

(a) For any  $0 < \delta \leq   \ep \leq \ep_0$ there exists a constant $R =  R (\delta , \ep) > 0$ such that 
$$\diam( \mt \cap B^u_T(z ,\ep))   \leq R \, \diam( \mt \cap B^u_T (z , \delta))$$
for any $z \in \mt$ and any $T > 0$.

\ms

(b) For any $\ep \in (0,\ep_0]$ and any $\rho \in (0,1)$ there exists $\delta  \in (0,\ep]$
such that for  any $z\in \mt$ and any $T > 0$ we have
$\diam ( \mt \cap B^u_T(z ,\delta))   \leq \rho \; \diam( \mt \cap B^u_T (z , \ep)) .$

\bs


{\sc  Standing Assumptions:} 

\ms

(A) $\varphi_t$ has Lipschitz local holonomy maps over $\mt$,

\ms  

(B) the local non-integrability condition (LNIC) holds for $\varphi_t$ on $\mt$,

\ms

(C) $\varphi_t$  has a regular distortion  along unstable manifolds over the basic set $\mt$.

\ms

In this paper we will work under the above Standing Assumptions\footnote{These assumptions are needed to ensure the validity of certain 
strong spectral estimates for Ruelle transfer operators (see \cite{PeS3}). However recent developments in \cite{St3} suggest
that similar estimates should be established in much higher generality. So, we expect that the methods of the present paper would
apply without change under much more general assumptions.}. 

A large class of examples satisfying the conditions (A) -- (C) is provided by imposing the following {\it pinching condition}:

\ms

\noindent
{\bf (P)}:  {\it There exist  constants $C > 0$ and $\beta \geq \alpha > 0$ such that for every  $x\in M$ we have
$$\frac{1}{C} \, e^{\alpha_x \,t}\, \|u\| \leq \| d\varphi_{t}(x)\cdot u\| 
\leq C\, e^{\beta_x\,t}\, \|u\| \quad, \quad  u\in E^u(x) \:\:, t > 0 $$
for some constants $\alpha_x, \beta_x > 0$ with
$\alpha \leq \alpha_x \leq \beta_x \leq \beta$ and $2\alpha_x - \beta_x \geq \alpha$ for all $x\in M$.}

\ms

It is well-known that  {\bf (P)} holds for geodesic flows on manifolds of strictly negative sectional curvature satisfying the so called 
$\frac{1}{4}$-pinching condition.  {\bf (P)} always holds when  $\dim(M) = 3$.

\ms

{ \bf Simplifying Assumptions:} $\varphi_t$ is a $C^2$ contact Anosov flow satisfying the condition {\bf(P)}.

\ms

It follows from the results in \cite{St2}, that the pinching condition  {\bf (P)} implies that $\varphi_t$ has  Lipschitz local holonomy maps 
and regular distortion along unstable manifolds and moreover:

\ms

{\sc The Simplifying Assumptions  imply the Standing Assumptions}.
 
 \ms

{\bf Throughout this paper we work under the Standing Assumptions}. 

\ms

\subsection{Some definitions and Ruelle transfer operators}

As in Section 1, let $\rr = \{R_i\}_{i=1}^k$ be a Markov family for $\varphi_t$ over $\Lambda$ consisting of 
rectangles $R_i = [U_i ,S_i ]$, where $U_i$ (resp. $S_i$) are (admissible) subsets of $W^u_{\ep}(z_i) \cap \mt$
(resp. $W^s_{\ep}(z_i) \cap \mt$) for some $\ep > 0$ and $z_i\in \mt$.
We will use the set-up and some arguments from \cite{St1} and \cite{PeS2}.  As in these papers, fix a 
{\it (pseudo) Markov family} $\rr = \{ R_i\}_{i=1}^k$ of {\it pseudo-rectangles} 
$$R_i = [U_i  , S_i ] =  \{ [x,y] : x\in U_i, y\in S_i\} .$$
Set 
$$R = \cup_{i=1}^{k} R_i \quad , \quad  U = \cup_{i=1}^k U_i .$$
Consider the
{\it Poincar\'e map} $\pp: R \longrightarrow R$, defined by  $\pp(x) = \varphi_{\tau(x)}(x) \in R$, where
$\tau(x) > 0$ is the smallest positive time with $\varphi_{\tau(x)}(x) \in R$ ({\it first return time}  function). 
The  {\it shift map}  $\sigma : U   \longrightarrow U$ is given by
$\sigma  = \piU \circ \pp$, where $\piU : R \longrightarrow U$ is the {\it projection} along stable leaves.

Recall the subsets $\hU$ of $U$ and $\hR$ of $R$ introduced in Sect. 1.
Throughout $\alpha > 0$ will be {\bf fixed constant}  such that  $\tau\in C^{\alpha}(\hU)$. We assume in Theorem 1 that $F, G \in C^{\alpha}(\hrt)$. 
If $\tau \in C^{\tilde{\alpha}}(\hU)$ for some $0 < \tilde{\alpha} < 1,$ we may take $\alpha = \min\{\alpha, \tilde{\alpha} \}$, to arrange that $F, G, \tau$ 
are in the H\"older spaces with the same $\alpha.$  Fir simplicity of the notations we assume in the following that this is arranged.
Since the local stable (and unstable) holonomy maps are uniformly H\"older 
(\cite{Ha1}, \cite{Ha2}), we may assume $\alpha$ is chosen so that
\be \label{eq:2.1} 
d([x,y], [x',y]) \leq C (d(x,x'))^\alpha,\ x,x', y \in R_i\:\:, \:\: i = 1,\ldots,  k .
\ee

The hyperbolicity of the flow  implies the existence of
constants $c_0 \in (0,1]$ and $\gamma_1 > \gamma > 1$ such that
\be \label{eq:2.2} 
c_0 \gamma^m\; (d (x,y))^{1/\alpha} \leq  d (\pp^m(x)), \pp^m(y)) \leq \frac{\gamma_1^m}{c_0} (d (x,y))^\alpha
\ee
for all $x,y\in R$ such that $\pp^j(x), \pp^j(y)$ belong to the same $R_{i_j}$ for all $j = 0,1, \ldots, m$. 
Moreover, we choose the constants so that if $y \in W^s(x)$ for some $x,y \in R_i$, then
\be \label{eq:2.3}
d(\pp^n(x), \pp^n(y)) \leq \frac{c_0}{\gamma^n} \quad, \quad n \geq 0 .
\ee
Fix a constant $\alpha' \in (0,\alpha]$ so that
$$\rho = \frac{\gamma_1^{\alpha'}}{\gamma} < 1 .$$

Let $\| h\|_0$ denote the {\it standard $\sup$ norm} of $h$ on $U$.  For  $|b| \geq 1$, and $\beta > 0$, as in \cite{D}, define the norm  
$$\|h\|_{\beta,b} = \|h\|_\infty + \frac{|h|_\beta}{|b|}$$ 
on the space $C^\beta(\hU)$ of $\beta$-H\"older functions on $\hU$.

As in \cite{PeS2} and \cite{PeS3},  in this paper we will frequently use {\it Ruelle transfer operators} of the form
$$L_{f - s\tau + zg} v(y) = \sum_{\sigma x = y} e^{f(x) - s \tau(x) + zg(x)} v(x),\:  s, z \in \C,\: y \in U ,$$
depending on two complex parameters $s$ and $z$. 
The following theorem was proved in \cite{PeS3}.

\ms

\begin{thm} \label{Thm3}
Let $\varphi_t : M \longrightarrow M$ satisfy the Standing Assumptions over the basic set $\mt$, and let $\alpha > \beta > 0$.
Let  $\rr = \{R_i\}_{i=1}^k$ be a Markov family for $\varphi_t$ over $\mt$ as above. Then 
for any real-valued functions  $f,g \in C^\alpha(\hU)$ and any constants $\nu > 0$ and $B > 0$ 
there exist constants $0 < \rho < 1$, $a_0 > 0$, $b_0 \geq 1$ and  $C = C(B, \nu)> 0$ 
such that if $a,c\in \R$ satisfy $|a|, |c| \leq a_0$ then
\begin{equation} \label{eq:2.4}
\|L_{f -(a+ \i b)\tau + (c+\i w) g}^m h \|_{\beta,b}  \leq C \,e ^{m \Pr_{\sigma}(f)}\, \rho^m \, |b|^{\nu}\, \| h\|_{\beta,b}
\end{equation}
for all $h \in C^\beta(U)$, all integers $m \geq 1$ and all $b, w\in \R$ with  $|b| \geq b_0$ and $|w| \leq B \, |b|$.
\end{thm}

\ms

\subsection{Sinai's lemma for suspension flows}

Given functions $G$ and $\tG$ on $\hrt$, define
$$g(x) = \int_0^{\tau(x)} G(x,t)\, dt,\:
\tg(x) =  \int_0^{\tau(x)} \tG(x,t)\, dt $$
for all $x \in \hR$. Here $G(x,t) = G(\pi(x,t)) = G(\sigma_t^\tau(x,0))$. It is easy to see that if $G \in C^\alpha(\hrt)$, then $g \in C^\alpha (\hR)$.

Consider the function (defined as in the proof of Proposition 1.2 in [PP])
\begin{equation} \label{eq:2.5}
p(x) =  \sum_{n=0}^\infty \left[g(\pp^n(x)) - g(\pp^n(\pi_U x))) \right], \quad x\in \hR .
\end{equation}
Since $x, \tx = \pi_U(x)$ belong to the same stable leaf in $\hR$, (\ref{eq:2.3}) implies
$$d (\pp^n(x), \pp^n(\pi_U(x))) \leq \frac{c_0}{\gamma^n} $$
for all $n \geq 0$. Thus, the series in (\ref{eq:2.5}) is convergent. Now define
\begin{eqnarray*} \tG(x,t) 
 =  G(\piU(x),t) \nonumber\\
 + \sum_{n=0}^\infty \Bigl(G(\pp^{n+1}(\piU(x))), t\, \tau(\pp^{n+1}(x))/\tau(x)) \nonumber \\
     - G(\pp^n(\piU(\pp(x))), t\, \tau(\pp^{n+1}(x))/\tau(x))\Bigr)  
\end{eqnarray*}
for $ x \in \hR$ and  $0 \leq t < \tau (x)$. Notice that since $\tau(x)$ is constant on stable leaves, we have
$$\tau( \pp^{n+1}(x)) = \tau(\pp^n(\piU(\pp(x))) .$$

\ms

The following is the analogue of the well-known Sinai's lemma (see e.g. Proposition 1.2 in \cite{PP}) for suspension flows.

\ms

\begin{prop}
$(a)$ {\it The function $p$  defined above belongs to  $C^\beta(\hR)$ for some $\beta > 0$ and
\begin{equation} \label{eq:2.6}
g(x) = \tg(x) + p(x) - p(\pp(x))
\end{equation}
for all $x \in \hR$. Moreover $\tG$ is constant on stable leaves of $\hrt$ and is $\beta$-H\"older, where }
\be \label{eq:2.7}
\beta = \alpha^2\alpha'/2 > 0 .
\ee 

\ms

$(b)$ {\it The function
\be \label{eq:2.8}
P(x,t) = \sum_{n=0}^\infty \left[ G(\pp^n(x), t\, \tau(\pp^n(x))/\tau(x) ) - G(\pp^n(\piU(x)), t\, \tau(\pp^n(x))/\tau(x) )\right]
\ee
$(x\in \hR$, $0 \leq t < \tau(x))$ is also $\beta$-H\"older on $\hrt$,
\be \label{eq:2.9}
p(x) = \int_0^{\tau(x)} P(x,t)\, dt, \quad x\in \hR ,
\ee
and
\be \label{eq:2.10}
G(x,t) = \tG(x,t) + P(x,t) - P(\pp(x), t\, \tau(\pp(x))/\tau(x) )
\ee
for all $x\in \hR$ and $0 \leq t  <\tau(x)$.}
\end{prop}

\ms

We prove Proposition 2.2 in the Appendix.

\ms

\begin{rem}
$(a)$ It follows from the definition of $g$ that it is $\alpha$-H\"older with $|g|_\alpha \leq C |G|_\alpha$. 
Then $(\ref{eq:2.5})$ and $d(\pp^n(x), \pp^n(\piU(x)) \leq c_0/\gamma^n$, which follows from $(\ref{eq:2.3})$, imply
$$|p(x)| \leq \sum_{n=0}^\infty C |G|_\alpha\; c_0/\gamma^n \leq C\, |G|_\alpha$$
for all $x \in R$, so $|p|_\infty \leq C\, |G|_\alpha$.  Similarly, $|P|_\infty \leq C\, |G|_\alpha$.

\ms

$(b)$ Given $y \in R_i$ for some $i$, consider the function $w_y (x) = h([x,y])$ on $R_i$. Now $(2.1)$ implies
$$|w_y(x) - w_y(x')| \leq |h|_\alpha \; (d([x,y], [x',y]))^\alpha \leq C |h|_\alpha\; (d(x,x'))^{\alpha^2} .$$
Thus, $w_y \in C_{\alpha^2}$ and $|w_y|_{\alpha^2} \leq C\, |h|_\alpha \leq C \, |g|_\alpha$. This can
be written as
$$|p([\cdot,y])|_{\alpha^2} \leq C\, |G|_\alpha,\: y \in R_i\:\:, \: i = 1, \ldots, k .$$
By $(\ref{eq:2.7})$ , $\beta < \alpha^2$, so the above is also true with $\alpha^2$ replaced by $\beta$.
With $(a)$ this gives $\|h([\cdot,y])\|_{\alpha^2} \leq C\, |G|_\alpha$ and so
\be
\|p([\cdot,y])\|_{\beta} \leq C\, |G|_\alpha,\: y \in R_i\:\:, \: i = 1, \ldots, k .
\ee

$(c)$ $($Following Ch. 1 in \cite{B2}; in particular sections 1B and 1C$)$ 

Let $\tf : R\longrightarrow \R$ depend only on $x \in U$, i.e. $\tf(x) = \tf(x')$ whenever
$\piU(x) = \piU(x')$. Then we can regard $\tf$ as a function on $U$, $\tf \in C^\beta(U)$, and by the 
Ruelle-Perron-Frobenius Theorem there exist (unique) positive function $h_0 \in C^\beta(U)$
and a probability measure $\nu$ on $U$ such that $(\lc_{\tf})^*\nu = \nu$ and $\int_U h_0\, d\nu = 1 $. Then  
\be
dm_{\tf}(x) = h_0(x)\, \nu(x)
\ee
 is a $\sigma$-invariant probability measure on $U$, called the
Gibbs measure determined by $\tf$. It gives rise to a $\pp$-invariant probability measure $\m$ on $R$ as
follows. Given a continuous real valued function $w$ on $R$, define
$$w^*(x) = \min \{ w(y) : y \in W^s_R(x) \},\: x \in R .$$
Then $w^*\in C(R)$ and $w^*$ is constant on stable leaves of $R$, so it can be considered as a function in $C(U)$. Define
$$\int_R w(z)\, d\m(z) = \int_U w^*(x) \, dm_{\tf(x)} .$$
As in section 1C in \cite{B2}, one checks that this defines a $\pp$-invariant probability measure $\m$ on $R$.
Usually one denotes $\m = dm_{\tf}$ and calls this the {\it Gibbs measure} determined by $\tf$ on $R$.
The above definition shows that if $w\in C(R)$ depends only on $x\in U$, i.e. $w(x) = w(x')$ whenever $\piU(x) = \piU(x')$, then we have
$$\int_R w(z)\, d\m(z) = \int_U w(x)\, d m_{\tf}(x) .$$

$(d)$ To estimate $\|e^{p([\cdot : y])} h_0\|_\beta$, first we have
$$|e^{p([\cdot , y])} h_0|_\infty \leq e^{C |G|_\alpha} |h_0|_\infty \leq C e^{C |G|_\alpha} .$$
Using $(a)$,
$$|e^{p([\cdot : y])} |_\beta \leq e^{C |G|_\alpha}\; |p([\cdot , y])|_\beta \leq C |G|_\alpha\; e^{C |G|_\alpha} .$$
This implies
$$|e^{p([\cdot , y])} h_0|_\beta \leq |e^{p([\cdot : y])}|_\infty\, |h_0|_\beta + |e^{p([\cdot : y])}|_\beta\, | h_0|_\infty \leq
C,$$
$$ e^{C |G|_\alpha} + C |G|_\alpha\; e^{C |G|_\alpha} \leq C |G|_\alpha\; e^{C |G|_\alpha} .$$
Combining the above estimates, yields
$$\|e^{p([\cdot , y])} h_0\|_\beta \leq C |G|_\alpha\; e^{C |G|_\alpha} $$
and this estimate is uniform in $y \in R$.
\end{rem}

\subsection{ Application of Proposition 2.2}

By Proposition 2.2 there exist functions $\tF(w, t) $ and $Y(w, t)$ such that
$$
F(w, t) = \tF(w, t) + Y(w, t) - Y\Bigl(\pp(w), \frac{t \tau(\pp(w))}{\tau(w)}\Bigr),$$
where $\tF(w, t)$ is constant on stable leaves. Let
$$\tf(w) = \int_0^{\tau(w)} \tF(w, t) dt, \quad  \: y(w) = \int_0^{\tau(w)} Y(w, t) dt .$$
By a change of variables $t \frac{\tau(\pp(w))}{\tau(w)} = s,$ one deduces
$$ \int_0^{\tau(w)}Y\Bigl(\pc(w), \frac{t \tau(\pc(w))}{\tau(w)}\Bigr)dt = y(\pp(w)) .$$

Now for the equilibrium state $m_F$ one obtains
$$m_F  = \frac{1}{\int \tau(u) d\mu_{\tf} } \Bigl(\mu_{\tf} (w)\times l\Bigr)$$
since 
$$\mu_{f(w)} = \mu_{\tf(w) + y(w) - y(\pc(w))} = \mu_{\tf}(w).$$ 
Hear $\mu_q$ denote the equilibrium state of $q$ which is a probability measure invariant with respect to $\pp$ and $l$ is the Lebesgue measure on $\R$.
Obviously, $\tf$ depends only on $x = \pi_U(w) \in U$. Moreover, adding a constant, we preserve $m_F$ and can arrange $\Pr(F) = 0$. Since 
$${\Pr_{\pc}} (f(w) - \Pr(F) \tau) = 0,$$
this implies
$$0  = {\Pr}_{\pc}(\tf(w) + y(w) - y(\pc(w)) = {\Pr}_{\pc}(\tf(w)).$$
We may express the pressure by
$$\Pr_{\pc}(\tf) = \lim_{n \to \infty}\log \sum_{\pc^n w = w} e^{\tf^n (w)}.$$
Since $\tf(w)$ depends only on $x = \pi_U w$  and $\pc^n w = w$ implies $\sigma^n x = x,$ one deduces $\Pr_{\pc}(\tf(x)) = \Pr_{\sigma}(\tf(x)) = 0.$
Therefore  the Ruelle operator $\lc_{\tf}$ has 1 as an eigenvalue with eigenfunction $h(x) > 0.$  Moreover, we have $d\mu_{\tf}(x) = h(x) d\nu(x)$ and $\nu(x)$ is a 
$\sigma$-invariant measure on $U$ which can be considered as a $\pp$-invariant measure on $R$ as we have mentioned above. For this measure we have
\begin{equation} \label{eq:2.13}
(\lc_{\tf}^*) ^n\nu(x) = \nu(x), \: n \geq 1.
\end{equation}

\section{Representation of the function $Z(s, \omega, a)$}
\renewcommand{\theequation}{\arabic{section}.\arabic{equation}}
\setcounter{equation}{0}

Let $\chi \in C_0^{\infty}(\R; \R^+)$ be a fixed cut-off function. Set $q_n(t) = e^{\xi(a) t}\chi_n(t),$ where 
$$\chi_n(t) = \chi\Bigl(\frac{t}{\ep_n}\Bigr),\:\ep_n = e^{-\ep n}, \: 0 < \ep \leq\mu_0/8,$$ 
 $\mu_0 >0$ being the constant introduced in Proposition 4.2 in Section 4.

 The Fourier transform of $\chi_n$ satisfies
$\widehat{\chi_n}(\omega) = \ep_n\, \hat{\chi}(e_n\, \omega).$  
Given a continuous function $Q$ on $\rt$, consider 
$$Q^T(\tilde{w}) =  \int_0^T Q(\vt_t(\tilde{w})) dt,\: \tilde{w} \in \rt.$$
 Notice that we have
$$\Pr_{\pp}(q - \Pr(Q)\tau) = 0$$
if $q(w) = \int_0^{\tau(w)} Q(w, t) dt, \: w \in R$ (see \cite{PP}).\\

To establish Theorem 1.3, we will study the nonnegative function 
$$\rho_n(T) := \int_{\rt} q_n((G^T - aT)(y)) d m_{F}(y).$$
We have
\begin{eqnarray*}
\rho_n(T) = \int_{\rt} q_n((G - a)^T(y))\, d m_F(y)  =  \int_{\rt} e^{\xi(a) (G -a)^T(y)}\, \chi_n((G - a)^T(y))\, dm_F(y)\\
=  \frac{1}{2\pi}\, \int_{\rt} \, \int_{\R} e^{\xi(a) (G - a)^T(y)}\, e^{\i \omega (G - a)^T(y)}\, \hat{\chi}_n(\omega)\, d m_F(y) \, d\omega\\
 =  \frac{1}{2\pi}\,  \int_{\R} \left( \int_{\rt} \, e^{z\, (G - a)^T(y)}\,d m_F(y)\right)\,  \hat{\chi}_n(\omega)\,  \, d\omega\;,
\end{eqnarray*}
where $z = \xi(a) + \i \omega$.

Define $\Gamma_z(T)$ for $T \geq 0$ by
$$\Gamma_z(T) = \int_{\rt} \, e^{z\, (G- a)^T(y)}\,d m_F(y)
$$ 
and $\Gamma_z(T) = 0$ for $T < 0$.
Our purpose is to study the Laplace transform 
\be \label{eq:3.1}
Z(s, \omega, a)
= \int_0^\infty e^{- sT}\,\Gamma_z(T)\, dT,\:s\in \C, \: \omega \in \R.
\ee

Since  for large $M_0 > 0$ we have $|\Gamma_z(T)| \leq C e^{C_1 T},\:\Re s >M_0$, with some constant $C_1 > 0$, the transformation $Z(s, \omega, a)$ exists for $\Re s > M_0$ 
uniformly with respect to $\omega \in \R$.  
In Section 4 we will show that $Z(s, \omega, a)$ has an analytic extension to
$$\{s \in \C: \: \gamma(a) - \mu_0 \leq \Re s, \omega \in \R\} \setminus \{(s(\omega,a):\: |\omega| \leq \ep_0\}, $$
with $\ep_0 > 0$ and $\mu_0 > 0$ sufficiently small. Here $s(\omega, a)$ is a 
simple pole described in Section 4.  We use the notation $Z(s, \omega, a)$ since the Laplace transform depends  on $s \in \C, \omega \in \R$ and $a$.

Set
$$f(w) = \int_0^{\tau(w)} F(w, t) dt,\:  g(w) = \int_0^{\tau(w)} G(w, t) dt,\: w \in R.$$
 We repeat the argument of Section 4 in \cite{W} (see also \cite{Po}) 
 to obtain a representation of $Z(s, \omega, a).$
For the equilibrium state $m_F$ of $F$ we apply the reduction in Subsection 2.4 and we obtain the measure $d\mu_{\tf}(x) = h(x) \nu(x)$, where $\tf(x)$ depends only on $x \in U$. 
For simplicity of the notation in the following we will denote $\tf(x)$ again by $f(x)$ and $\mu_{\tf}$ by $\mu_f$. 
Moreover, in  the following we assume that $G$ {\bf is constant on the stable foliations} in $R^{\tau}$, so $g(x)$ depends only on $x \in U.$ 

 Given a H\"older function $Q(w)$ on $\rt$, we have

$$\int_{R^{\tau}} Q(w) d m_F(w) = \frac{ \int_R \int_0^{\tau(x)} Q(x,\eta) d\eta d \mu_q(x)}{\int\tau d \mu},$$
where $\mu_q$ is the equilibrium state of $q(w) =\int_0^{\tau(w)} Q(w, t) dt.$
Therefore, setting $Q = z(G- a)$, we obtain
$$Z(s, w, a) = \frac{1}{\int\tau d\mu} \int_0^{\infty} e^{-(s + az)T} \Bigl(\int_R\int_0^{\tau(x)} e^{zG^T(x, \eta)} d\eta d\mu(x)\Bigr) dT$$
$$= \frac{1}{\int \tau d\mu} \int_0^{\infty} e^{-(s + az)T} \Bigl(\int_U \int_0^{\tau(x)} e^{zG^{T+ \eta}(x, 0)- zG^{\eta}(x, 0)} d\eta h(x)d\nu(x)\Bigr) dT.$$
 Here we interpret the integral on $R$ as an integral on $U$ as we have mentioned in Remark 1(c) in Section 2.
Given $T > 0,x \in U, 0 \leq\eta \leq \tau(x),$ there exists a unique choice of $n \geq 0$ and $0 \leq v < \tau(\sigma^n x)$ so that
$T + \eta = v + \tau^n(x)$. Notice that when $x \in U$ changes the integer $n$ may change but since
$$T - \tau(\sigma^n x) \leq \tau^n(x) = T + \eta - v \leq T + \tau(x),$$
we deduce that for fixed $T, \:T \geq T_0$, and all $x \in U$ there is only a finite number (depending on $T$) of possible choices for $n$.

For $T + \eta = \tau^n(x) + v$ one applies the formula
$$e^{G^{T + \eta}(x, 0)} = \sum_{n = 0}^{\infty} \int_0^{\tau(\sigma^nx)} e^{G^{v + \tau^n(x)}(x, 0)} \delta( \eta + T - v - \tau^n(x)) dv,$$ 
where for fixed $x \in U,\: T \geq T_0$ only one term in the infinite sum is not vanishing (see \cite{Po}, \cite{W} for a similar argument).
Then we may transform the integral 
$$\frac{1}{\int \tau d\mu} \int_0^{\infty} e^{-(s + az)T} \Bigl(\int_U \int_0^{\tau(x)} e^{z G^{T+ \eta}(x, 0)- z G^{\eta}(x, 0)} d\eta h(x)d\nu(x)\Bigr) dT$$
in the above expression for $Z(s, \omega, a)$, as in Section 4 in \cite{W} to obtain the representation

\begin{eqnarray*}
Z(s, \omega, a) = \frac{1}{\int \tau d\mu} \sum_{n=0}^{\infty}\int_U e^{-(s + (\xi(a) + \i \omega)a)\tau^n(x) + (\xi(a) + \i \omega)g^n(x)}\\
\times B_1(s, \omega, a, \sigma^n(x)) B_2(s, \omega, a, x) h(x)d \nu(x),
\end{eqnarray*}
where
$$B_1(s, \omega, a, x) = \int_0^{\tau(x)} \exp\Bigl( -(s+ az) v + z G^v(x, 0)\Bigr) dv,$$
$$ B_2(s, \omega, a,  x) = \int_0^{\tau(x)}\exp\Bigl((s + a z)\eta - z G^{\eta}(x, 0)\Bigr) d \eta.$$

We apply (\ref{eq:2.13}) and then use the adjoint of the Ruelle operator $\lc_{f}^*$, noting that
$$\Bigl[\lc_{f}^n \Bigl(e^{-(s + (\xi(a) + \i \omega) a)\tau^n + (\xi(a) + \i \omega)) g^n} d\Bigr)\Bigr](y)$$
$$= \Bigl[\lc^n_{f - s \tau + (\xi(a) + \i \omega)(g - a \tau)} d\Bigr](y). $$
Therefore, we conclude that
\begin{eqnarray} \label{eq:3.2}
Z(s, \omega, a) = \frac{1}{\int \tau d\mu} \sum_{n = 0}^{\infty}\int_U B_1(s, \omega, a, y) \nonumber \\
\times\Bigl({\mathcal L}^n_{f -(s  \tau + (\xi(a) +\i \omega)(g -  a \tau)}
\Bigl[h(.)B_2(s, \omega, a, .)\Bigr]\Bigr)(y) d \nu(y).
\end{eqnarray}

\section{Meromorphic extension of $Z(s, \omega, a)$}
\renewcommand{\theequation}{\arabic{section}.\arabic{equation}}
\setcounter{equation}{0}

 We assume $f(x)$ and $g(x), \: x \in U$, fixed as in Section 3. 
Introduce the Ruelle operator

$$\lc_{s, \omega, a} = \lc_{f  - s \tau + (\xi(a) + \i \omega)(g- a\tau) },\: s \in \C, \: \omega \in \R.$$

It is easy to see that for $s = \gamma(a)$ and  $\omega = 0$ we have
$$\Pr_{\pc} (f +  \xi(a) (g -a \tau) -\gamma(a)\tau) = 0.$$
Indeed,
$$\gamma(a) = \Pr\:(F + \xi(a) G) - \xi(a) a $$
and
$$\Pr_{\pc}(f + \xi(a) g- \Pr\:(F + \xi(a) G)\tau) = 0.$$
Notice that there is an unique number $t$ such that
$$\Pr_{\pc}(f +\xi(a)(g - a \tau) - t \tau) = 0.$$
Since $f, g, \tau$ depend only on $x \in U$, as in Subsection 2.4 we deduce that
$$\Pr_{\sigma} (f + \xi(a) (g - a \tau) - \gamma(a) \tau) = 0.$$
 Below we will write simply $\Pr$ instead of $\Pr_{\sigma}$ if there are no confusions.

Set $f_a =f + \xi(a) (g - a \tau)$ and  consider the Ruelle operator
$\lc_{s, \omega, a},$ 
where $s = \gamma(a) + q + \i b,\: q \in \R, \: b \in \R,\: \omega \in \R$. 
 Let
$$p(s, w, a) = f_a - s \tau + \i \omega (g - a \tau),\: s = \gamma(a) + q +\i b.$$
Since $\Pr \:p(\gamma(a), 0, a) = 0,$ by a standard argument we may define $\Pr\: p(s, \omega, a)$ for $(s, \omega)$ in a small neighbourhood of $(\gamma(a), 0)$ in $\C^2$ (see \cite{PP}). 
Recall the following result proved in \cite{W}.

\begin{prop} [Proposition 4, \cite{W}] Let $G^{\alpha}(\rt)$ be a function such that $G$ and $\vt_t$ are flow independent. Assume that $G$ is constant on stable leaves. Then \\
$(i)$ The function $Z(s, \omega, a)$ is analytic for $(s, \omega) \in \{s \in \C:\: \Re s > \gamma(a)\} \times \R.$\\
$(ii)$ There exists an open neighbourhood $W$ of $(\gamma(a), 0)$ in $\C^2$ such that for $(s, \omega) \in W$ we have
\begin{equation} \label{eq:4.1}
Z(s, \omega, a) = \frac{B_3(s, \omega,a)}{ 1 - \exp\Bigl(\Pr(p(s, \omega, a))\Bigr)} + J(s, \omega, a),
\end{equation}
where
\begin{equation} \label{eq:4.2}
B_3(s, \omega, a) = \frac{1}{\int\tau d\mu} \int_ R B_1(s, \omega ,a, .) h_{p(s, \omega, a)} (x)d\nu(x) \int h B_2(s, \omega, a, .)d \nu_{p(s, \omega, a)}.
\end{equation} 
and $J(s, \omega, a)$ is analytic for $(s, \omega) \in W.$ Here $h_{p(s, \omega, a)}(x) > 0$ is the eigenfunction corresponding to the eigenvalue $e^{\Pr(p(s, \omega, a))}$ of $\lc_{p(s, \omega, a)}$ 
and similarly the measure $\nu_{p(s, \omega, a)}$ is determined by the eigenvalue $e^{\overline{\Pr(p(s, \omega, a))}}$ of the dual operator $\lc^*_{p(s,\omega a)}.$\\
$(iii)$ $Z(s, \omega, a)$ is analytic for $(s, \omega)$ in an open neighbourhood $V_1$ of $\{s: \Re s = \gamma(a), \: s\neq \gamma(a))\} \times \{0\}.$\\
$(iv)$ For each $\omega \in \R \setminus \{0\}$, $Z(s, \omega, a)$ is analytic for $(s, \omega)$ in an open neighbourhood $V_2$ of $\{s: \Re s = \gamma(a)\} \times \{\omega\}.$
\end{prop}

For our analysis we need to estimate the norms 
$$\|B_1(s, \omega, a,x)\|_{\infty},\:\| B_2(s,\omega, a,.)\|_{\beta}$$

The norm $\|B_2(s, \omega, a, .)\|_{\infty}$ is easily estimated uniformly with respect to $\omega\in \R$,
since
$$\big|\exp\Bigl((s + a(\xi(a) + \i \omega))\eta - (\xi(a) + \i \omega) G^{\eta}(x, 0)\Bigr)\big| $$
$$\leq \exp \Bigl((|\Re s| + a|\xi(a)|)\eta + |\xi(a)||G^{\eta}(x, 0)|\Bigr)$$
and 
$$\|B_2(s, \omega, a, x)\|_{\infty} \leq \exp \Bigl((|\Re s| + a|\xi(a)|)\kappa_1 + |\xi(a)|\max_{x \in U} G(x, 0) \kappa_1\Bigr),$$
where $\kappa_1 = \max_{x \in U}\tau(x).$
Similarly, one treats the norm  $\|B_1(s, \omega, a,.)\|_{\infty}.$
For the norm $|B_2(s, \omega, a, .)|_{\beta}$ we apply the following elementary estimate. Let
$$k(x) = \int_0^{\tau(x)}e^{(s+az)\eta} e^{K^{\eta}(x, 0)} d\eta,\: K \in C^{\beta}(\rt).$$
Then 
\begin{eqnarray*}
|k(x) - k(y)| = \big |\int_0^{\tau(x)}e^{(s + a z)\eta} e^{K^{\eta}(x, 0)} d\eta - \int_0^{\tau(y)}e^{(s + a z)\eta} e^{K^{\eta}(y, 0)} d\eta\big |\nonumber \\
\leq \int_0^{\tau(x)} e^{|\Re s + a \xi(a)|\eta}\Bigl|e^{K^{\eta}(x, 0)} - e^{K^{\eta}(y, 0)}\Bigr| d\eta + \big|\int_{\tau(x)}^{\tau(y)} e^{(s + a z)\eta} e^ {K^{\eta}(y, 0)} d\eta \big|.
\end{eqnarray*}
The second term in the right-hand-side is estimated by 
$$\exp(\max_{0 \leq \eta \leq \kappa_1}|\Re K^{\eta}| + |\Re s + a\xi(a)|\kappa_1) |\tau(x) - \tau(y)|.$$

 For the first term on the right we use the inequality 
$$|e^{z_1} - e^{z_2}| = \big|\int_{z_1}^{z_2} e^u du\big| \leq e^{ |\Re z_1| + |\Re z_2|} |z_1 - z_2|,$$
and we obtain
$$\int_0^{\tau(x)}e^{|\Re s + a \xi(a)|\eta}\Bigl |e^{K^{\eta}(x, 0)} - e^{K^{\eta}(y, 0)}\Bigr| d\eta \leq \exp\Bigl(( 2 \|\Re K\|_{\infty} + |\Re s + a\xi(a)|)\kappa_1\Bigr)$$
$$ \times \int_0^{\tau(x)}|K^{\eta}(x, 0) - K^{\eta} (y, 0)| d\eta$$
$$\leq \exp\Bigl(( 2 \|\Re K\|_{\infty} + |\Re s + a\xi(a)|)\kappa_1\Bigr) \int_0^{\tau(x)} \int_0^{\eta}\Bigl| K(\sigma^{\tau}_t (x, 0)) - K(\sigma^{\tau}_t (y, 0))\Bigr| dt d\eta$$
which yields an estimate for $|B_2(s, \omega, a, .)|_{\beta}$ uniformly with respect to $\omega \in \R.$

 For $(s, \omega) = (\gamma(a), 0)$ one has a maximal real eigenvalue 1 of $\lc_{(\gamma(a),0, a)}$  and the rest of the spectrum is contained in a disk of radius $0 < r < 1$. 
 By perturbation theory there exists an unique eigenvalue with maximal modulus of $\lc_{s,\omega, a}$ given by
  $$\lambda_{s, \omega, a} = \exp\Bigl(\Pr (p(s, \omega, a))\Bigr) ,$$
 defined for $(s, \omega) \in W.$ 
We get
$$\frac{\partial \lambda_{s, \omega, a}}{\partial s}\bigg \vert_{(s, \omega, a) = (\gamma(a), 0, a)} = - \int \tau d \mu_{f_a} <  0,$$
where $\mu_{f_a}$ is the equilibrium state of $f_a$.
By the implicit function theorem (see Lemma 3 in \cite{W}) for small $\epsilon_1 > 0$ we may determine $s = s(\omega, a),\: |\omega| \leq \ep_0,$ from the equation $\lambda_{s, \omega, a} = 1$ so that 
$$\lambda_{(s(\omega, a), \omega, a)} = 1, \: s(0, a) = \gamma(a).$$
Therefore 
$$\frac{s - s(\omega, a)}{1 - \exp\Bigl(\Pr(p(s, \omega, a))\Bigr)} = \Bigl(\int \tau d\nu_{f_a -s(\omega, a) \tau + \i \omega g}\Bigr)^{-1} + {\mathcal O} (s- s(\omega, a)).$$
This shows that we have a pole at $s = s(\omega, a)$ and taking the residue at $s(\omega, a)$, the singular term in (\ref{eq:4.1}) becomes
$$ \Bigl(\frac{B_3(s(\omega, a), \omega,a)}{\int \tau d\nu_{f_a -s(\omega, a) \tau + \i \omega g}}\Bigr)\frac{1}{s- s(\omega, a)}.$$

Now we will show  that $Z(s, \omega, a)$ has an meromorphic continuation across the line $\Re s = \gamma(a)$.
First note that for $a = \int_U G(y) dm_{F + \xi(a) G}$ with a H\"older function $G \geq g_0 > 0$ on $U$ one has
$$g_0 \leq a \leq \max_{y \in U} G(y) = m.$$
 Let $0 < \eta < g_0/2$ be a fixed number. We will apply the spectral estimates for the operator $\lc_{s, \omega, a}$ given in Theorem 2.1 in Section 2 (see \cite{PeS3}). It is possible to write 
 $\lc_{s, \omega, a}$ in two different forms
\begin{eqnarray*}
\lc_{1,s, \omega, a}= \lc_{h_a - (\Re s -\gamma(a))\tau- \i \Im s \tau + \i \omega(g - a \tau)},\\
\lc_{2,s, \omega, a}= \lc_{h_a - (\Re s- \gamma(a))\tau -\i (\Im s + a \omega) \tau + \i \omega g},
\end{eqnarray*}
where $h_a = f_a - \gamma(a) \tau$ and $\Pr_{\sigma}(h_a) = 0.$ In the operators $\lc_{k,s,\omega, a},k = 1,2,$ we have different factors $-\i \Im s $ and $-\i (\Im s + a \omega)$ in front of $\tau.$ 
Applying Theorem 2.1, we can find $a_0 > 0$ and constants
$0 < \rho < 1, \: M > 0$ such that for $|\Re s - \gamma(a)| \leq a_0$ and any $\nu > 0$ we have
\begin{equation} \label{eq:4.3}
\|\lc_{1, s, \omega, a}^m h\|_{\beta, \Im s} \leq C(\nu, B_1) \rho^m |\Im s|^{\nu} \|h\|_{\beta, \Im s}
\end{equation}
for $|\Im s | \geq M,\:|\omega| \leq B_1|\Im s|,$
\begin{eqnarray}\label{eq:4.4}
\|\lc_{2, s, \omega, a}^m h\|_{\beta, (\Im s + a\omega)} \leq D(\nu, B_2) \rho^m |\Im s + a \omega|^{\nu} 
\|h\|_{\beta, \Im s + a \omega}
\end{eqnarray}
for $\:|\Im s + a \omega| \geq M,\:|\omega| \leq B_2|\Im s + a \omega|.$
Let us remark that we can take the same constants $a_0, \rho$ and $M$ in both estimates above, since if we have constants $$a_k > 0,\:0 < \rho_k < 1, M_k > 0, k = 1, 2$$
for the operators $\lc_{k,s,\omega, a},$  we can choose 
$$a_0 = \min\{a_1, a_2\},\:\rho = \max\{\rho_1, \rho_2\},\: M = \max\{M_1, M_2\}.$$
On the other hand, the constants $C(\nu, B_1)$ and $D(\nu, B_2)$ 
depend  on $(\nu, B_1)$ and $(\nu, B_2)$, respectively.

\medskip

\begin{prop}
Assume  the assumptions of Theorem $1.3$ fulfilled. Then for any  H\"older continuous functions
$F, G \in C^{\alpha} (\hrt)$ there exist $\mu_0 > 0$ and $\ep_0 > 0$ such that the function $Z(s, \omega, a)$ admits a meromorphic continuation for 
\begin{equation} \label{eq:4.5}
(s, \omega) \in \{(s, \omega) \in \C^2:\: \Re s \geq \gamma(a) - \mu_0,\: \omega \in \R\}
\end{equation}
with only one simple pole at $s(\omega, a),\: |\omega | \leq \ep_0.$ The pole $s(\omega, a)$ is determined as the root of the equation $\Pr(f_a - s \tau + \i\omega (g- a \tau)) = 0$ with respect to $s$ for $|\omega| \leq \ep_0.$ Moreover, there exist constants $\eta > 0,$ $M > 0$  such that for any $\nu > 0$ if $|\Im s| \geq M$ or $ |\omega| \geq \frac{1}{\eta}M$, we have the estimate
\begin{equation} \label{eq:4.6}
|Z(s, \omega, a)| \leq B_{\nu} (|\Im s| ^{\nu} + |\omega|^{\nu}), \: \Re s \geq \gamma(a) - \mu_0, 
\end{equation}
uniformity with respect to $a \in J$ in a compact interval $J \Subset \Gamma_G$ with a constant $B_{\nu} >0$ independent on $s, \omega$ and $a \in J.$
\end{prop}

\begin{proof} 

 We suppose below that $|\Re s -\gamma(a) |\leq a_0$, since for $\Re s > \gamma(a) + a_0$ the result follows from Proposition 4.1, (i).  Consider three cases.\\

{\bf Case 1.} $(\Im z, \omega) \in D_M = \{|\Im z| \leq M, \:|\omega| \leq \frac{1}{\eta}M\}.$\\
   For $(\Im s, \omega) \in B_{\ep_0} = \{|\Im z| <\ep_0,|\omega|< \ep_0\}$ the result follows from Proposition 4.1, (ii). So assume that
$(\Im z, \omega) \in D_M \setminus B_{\ep_0}.$

In this situation we may apply the statement (iii) of Proposition 4.1. For reader's convenience we present a proof. Let $(s_0,\: w_0) $ with $(\Im s_0,\omega_0) \in D_M\setminus B_{\ep_0}$ be fixed. 
Assume first that $\Im p(s_0, \omega_0, a)$ is cohomologous to $c + 2\pi Q$ with an integer-valued function $Q \in C(U; \Z)$ and a constant $c \in [0, 2 \pi ).$ Then we define the pressure 
$\Pr(p(s_0, \omega_0, a)) = \Pr(f_a) + c$ and we extend the pressure  in a small neighbourhood of $(s_0,\omega_0)$. Since $G$ and $\sigma^{\tau}_t$ are flow independent, the functions $g$ 
and $\tau$ are $\sigma$-independent. If we have $c = 0$, from the fact that $\Im s_0 \tau + \omega_0 g$ is cohomologous to a function in $C(U; 2 \pi \Z),$ we deduce $\Im s_0 = \omega_0 = 0$ which 
is impossible. Thus we have $c \neq 0.$ Consequently, the operator $\lc_{s_0, \omega_0, a}$ has an eigenvalue $ e^{\i c}$. Then there exists a neighborhood $U_2$ of $(s_0, \omega_0)$ such that for $(s, \omega) \in U_2$ 
we have $\Pr(p(s, \omega, a)) \neq 0$ and for $(s, \omega) \in U_2$ we have an analytic extension of $Z(s, \omega, a)$ given by
$$Z(s, \omega, a) =  \Bigl[\frac{B_4(s, \omega,a)}{1 - e^{\Pr(p(s, \omega, a))}} + J_{2}(s, \omega, a)\Bigr]$$
with a function $J_{2}(s, \omega, a)$ analytic with respect to $s$ for $(s, \omega) \in U_2.$ Second, let $\Im p(s_0, \omega_0, a)$ be not cohomologous to $c + 2 \pi Q$. Then the spectral radius of 
$\lc_{s_0, \omega_0, a}$ is strictly less than 1 and this will be the case for $(s, \omega)$ is a small neighbourhood $U_3$ of $(s_0, \omega_0).$ Therefore it is easy to see that the series in (\ref{eq:3.2}) is 
absolutely convergent and we obtain again an analytic extension. Covering the compact set $D_M \setminus B_{\ep_0}$ by a finite number of neighbourhoods, we may choose $\mu_0 > 0$ small so that for 
$\gamma(a) - \mu_0 \leq \Re s \leq \gamma(a)$, $(\Im s,\omega)\in D_M \setminus B_{\ep_0}$, we have an analytic extension of $Z(s, \omega, a)$ in (\ref{eq:4.5}).\\

{\bf Case 2.} $|\Im s|\leq M,\:|\omega| > \frac{1}{\eta} M > \frac{2}{g_0}M.$\\

 Notice  that $\frac{2M}{g_0} \geq \frac{2M}{a}.$ We consider the operator $\lc_{2, s,\omega, a}$ and observe that
$$|\Im s + a \omega| \geq a|\omega| - |\Im s| \geq M$$
and also 
$$|\Im s + a \omega| \geq \frac{a}{2}|\omega| + \frac{a}{2}|\omega| - |\Im s| \geq \frac{a}{2}|\omega|.$$
Hence $|\omega| \leq \frac{2}{a} |\Im s + a \omega|\leq \frac{2}{g_0}|\Im s + a \omega|.$

To apply the estimate (\ref{eq:4.4}) with $B_2 = \frac{2}{\delta_0}$ to the series in (\ref{eq:3.2}), we must estimate the norm
$$\|h B_2(s, \omega, a, .)\|_{\beta, \Im s + a\omega}$$
uniformly with respect to $\omega \in \R.$ The norm $\|B_2(s,\omega, a,.)\|_{\beta}$ has been estimated above. Next one gets
$$\Bigl|\frac{1}{\Im s + a \omega}\Bigr| \leq \frac{1}{M}$$
and we deduce the needed estimate.

Now the series in (\ref{eq:3.2}) is absolutely convergent and we obtain an analytic extension of $Z(s, w, a)$ for $|\Re s - \gamma(a) | \leq a_0,\: |\Im s | \geq M,\: |\omega| >\frac{1}{\eta}M$ as well as the  estimate
\begin{equation} \label{eq:3.7}
|Z(s, \omega, a)| \leq C_{\nu, B_2} |\Im s + a \omega|^{\nu} \leq C_{\nu, B_2, M}(1 + |\omega|)^{\nu}.
\end{equation}
  Decreasing, if is necessary, $\mu_0$ we obtain an analytic extension in (\ref{eq:4.5}).\\

 {\bf Case 3.} $|\Im s | > M.$ 

We consider two subcases:\\
 
{\bf Subcase 3a.} $|\omega|\leq \frac{1}{\eta} |\Im s|.$ We work with the operator $\lc_{1, s, \omega, a}$. For every $\nu > 0$ with $b = \Im s, \: B_1 = \frac{1}{\eta}$ and $|\Im s | \geq M$ one obtains from 
(\ref{eq:4.3}) the spectral estimates
\begin{equation} \label{eq:4.8}
\|\lc_{s,\omega, a}^m h\|_{\beta, b}\leq C(\nu, \eta)\rho^m |\Im s|^{\nu}\|h\|_{\beta, b},\:m \in \N.
\end{equation}
The series (\ref{eq:3.2}) is absolutely convergent  and we deduce
$$|Z(s,\omega, a)| \leq \tilde{C}(\nu,\eta) |\Im s|^{\nu}.$$

{\bf Subcase 3b.} $|\omega| > \frac{1}{\eta} |\Im s|.$
We work now with the operator $\lc_{2,s,\omega, a}$. In this case
$$|\Im s + a\omega| \geq a|\omega| - |\Im s| \geq (a- \eta) |\omega| \geq \frac{1}{2}g_0|\omega|.$$
Hence
$$|\omega| \leq \frac{2}{g_0} |\Im s + a \omega|.$$
On the other hand, for $|\omega |>\frac{1}{\eta}|\Im s| \geq \frac{1}{\eta}M $ we  have
$$|\Im s + a \omega| \geq \frac{g_0}{2} |\omega| \geq \frac{g_0}{2 \eta}M \geq M,
$$ 
because $2\eta <  g_0.$
Therefore, for any $\nu >0$ we can use (\ref{eq:4.4}) with $B_2 = \frac{2}{g_0}$ and obtain
\begin{eqnarray} \label{eq:4.9}
\|\lc_{s,\omega, a}^m h\|_{\beta, \Im s + a \omega}\leq D(\nu, g_0)\rho^m |\Im s + a \omega|^{\nu}\|h\|_{\beta, \Im s + a \omega} \nonumber \\
\leq D(\nu, g_0, M)\rho^m (1 + |\omega|)^{\nu}\|h\|_{\beta, \Im s + a \omega}, \:m \in \N.
\end{eqnarray}
These estimates lead again to (\ref{eq:3.7}) and the proof is complete.
\end{proof}

\section{Tauberian theorem}   
\renewcommand{\theequation}{\arabic{section}.\arabic{equation}}
\setcounter{equation}{0}

In this section we prove a Tauberian theorem for a sequence of functions $\{g_n(t)\}_{n \in \N}$ similar to that in \cite{KS} (see also Proposition 6 (i) in \cite{W}). The novelty here is 
that the leading terms contain a factor $A_n \geq e^{-\mu n}$ 
with $\mu > 0$ which can converge to 0 exponentially fast and the remainders must be smaller that this leading term. Moreover, we have two limits $n\to \infty$ and $t \to \infty$ and this creates 
new difficulties. Under some assumptions on the Laplace transform of $g_n(t)$, stronger that those in \cite{KS}, we obtain an asymptotic  for $t \geq n-q$ and $n \to \infty.$

\begin{prop}   Let $g_n(t),\: n \in \N,$ be monotonic nondecreasing functions defined for $t \in [0, \infty)$ such that $g_n(0) = 0, \: n \in \N$. Assume that for any $n \in \N$ the Laplace transform
$$F_n(s) = \int_0^{\infty} e^{-s t} g_n(t)dt$$
is analytic for $\Re s > 1$. Assume that there exist $\mu_0 >0$ and $M > 0, C_0 > 0, C_1 > 0, \delta_0 >0$ such that $F_n(s)$ has a representation
\begin{equation} \label{eq:5.1}
F_n(s) =\frac{A_n}{\sqrt{s - 1}} + A_n K_n(s) + L_n(s),
\end{equation}
where $C_0 e^{-\mu n} \leq A_n \leq C_1, \: 0 < \mu \leq \mu_0/4, \: \forall n \in \N$, $K_n(s)$ are analytic  functions for $\Re s > 1$ and for $t \in \R$ and $s = 1 + \delta + \i t, 0 < \delta \ll 1,$ 
$K_n(s)$ has a limit $k_n(1+\i t) \in W^{1,1}_{loc}(\R)$ $($functions which are locally integrable and have locally integrable derivatives$)$ almost everywhere on $\R$ as $\delta \searrow 0$ and 
$|k_n(1+ \delta + \i t)| \leq k_0(t),$ where $k_0(t)$ is locally integrable, while
$L_n(s)$ has an analytic continuation for  $1- \mu_0 \leq \Re s \leq 1 +\delta_0.$ 
   Moreover, assume that the functions $K_n(s)$ have analytic continuations to $1 - \mu_0 \leq \Re s \leq 1 + \delta_0,\: |\Im s | \geq M$, 
and for every compact set $D \subset \R$ uniformly with respect to $n$ we have
$$\|k_n'(1 + \i t)\|_{L^1(D)} \leq C(D),\:\forall n \in \N.$$

Next assume that for any $0 < \nu < 1$ with a constant $B(\nu) > 0$ independent on $n$ we have for any $n \in \N$ the estimates
\begin{equation} \label{eq:5.2}
\Bigl|\frac{d^k}{ds^k}L_n(s)\Bigr|\leq B(\nu)( 1 +|\Im s|^{\nu}), \:1 - \mu_0 \leq  \Re s \leq 1 +\delta_0,\: k = 0, 1,
\end{equation}
\begin{equation}\label{eq:5.3}
\Bigl|\frac{d^k}{ds^k}K_n(s)\Bigr| \leq B(\nu)(1 + |\Im s|^{\nu}), \: 1 - \mu_0 \leq  \Re s \leq 1 +\delta_0,\:k = 0, 1,\ |\Im s| \geq M.
\end{equation}
 Then for fixed $ q \geq 0$ and for any $0 <\eta \ll 1$ there exists $n_0(\eta) \in \N$ such that for $t \geq  n- q$ and $n \geq n_0(\eta) + q$ we have
\begin{equation} \label{eq:5.4}
\frac{A_n e^{t}}{\sqrt{\pi t}}(1 -\eta) \leq g_n(t) \leq \frac{A_n e^{t}}{\sqrt{\pi t}}( 1 +\eta).
\end{equation}
\end{prop} 
\begin{rem}  Notice that if the estimates $(\ref{eq:5.2})$, $(\ref{eq:5.3})$ are satisfied for $\mu_0 > 0$ and $k = 0,$ then fixing  numbers $0 < \mu_1 < \mu_0, \: 0 < \delta_1 < \delta_0,$ we obtain the 
estimates $(\ref{eq:5.2})$ $($resp.$(\ref{eq:5.3}) )$ for $1 - \mu_1 \leq \Re \leq 1 + \delta_1,$ $($resp. for $1 - \mu_1 \leq \Re s \leq 1 + \delta_1, \: |\Im s| \geq M)$ with another constants
$B_1(\nu)$ by applying the Cauchy formula for the first derivative of the analytic functions $K_n(s)$ and $L_n(s).$ In the proof we use only the estimates with $\nu = 1/3$, but if we change the relation 
between $\mu, \mu_0$ and $t \geq \beta(n)$, we need estimates with different $\nu$. In Section {\rm 6} we establish $(\ref{eq:5.2}),\:(\ref{eq:5.3})$.
\end{rem}
{\it Proof.}  We follow the proof in \cite{KS} with a more precise analysis concerning the dependence of $n$. If we replace the function $g_n(t)$ by $\tilde{g}_n(t)=g_n(t)\sqrt{t}$, one obtains
$$F_n(s) = \int_0^{\infty} \frac{1}{\sqrt{t}}e^{-st} \tilde{g}_n(t)dt.$$

 For simplicity of the notations below we will denote  $\tilde{g}_n(x)$ by $g_n(x)$.  Now introduce the function $H_n(x) = g_n(x) e^{-x}$ and for $s = 1 + \ep + \i t$ define
$$\kc_{\ep, n}(t) = F_n(s) - \frac{A_n}{\sqrt{s-1}}.$$
After a change of variable $x = v^2$, one obtains
$$\kc_{\ep, n}(t) =2 \int_0^{\infty} e^{-(s-1)v^2} \Bigl(H_n(v^2) - \frac{A_n}{\sqrt{\pi}}\Bigr) dv.$$
Consequently,
$$\kc_{\ep, n}(t) =\lim_{\xi \to \infty} 2\int_0^{\xi} \Bigl(H_n(v^2)-\frac{A_n}{\sqrt{\pi}}\Bigr) e^{-(s-1)v^2} dv$$
and for fixed $n$ and fixed $\ep$ this limit is uniform for $|t| \leq 2 \lambda$. We multiply $\kc_{\ep, n}(t)$ by $\sqrt{y} e^{\i t y} \Bigl(1 - \frac{|t|}{ 2\lambda}\Bigr)$ and integrate over $t$ in $[-2\lambda, 2\lambda].$ Thus
$$\int_{- 2\lambda}^{2 \lambda} \sqrt{y} \Bigl(1 - \frac{|t|}{2 \lambda}\Bigr) \kc_{\ep,n}(t)e^{\i y t} dt$$
$$= \lim_{\xi \to \infty} 2\int_{-2\lambda}^{2 \lambda} \sqrt{y} e^{\i ty}\Bigl(1 - \frac{|t|}{2\lambda}\Bigr)\Bigl(\int_0^{\xi} \Bigl(H_n(v^2)-\frac{A_n}{\sqrt{\pi}}\Bigr) e^{-\ep v^2- \i t v^2} dv\Bigr) dt.$$

As in \cite{KS}, we interchange the limit $\xi \to \infty$ and the integration and write the right hand side of the last term as$$2 \int_0^{\infty} \Bigl(H_n(v^2)-\frac{A_n}{\sqrt{\pi}}\Bigr)e^{-\ep v^2} \Bigl(\int_{-2}^{2} \lambda\sqrt{y} e^{\i \lambda(y - v^2) u}\Bigl(1 - \frac{|u|}{2}\Bigr)du\Bigr) dv.$$

We change the variable $v = \sqrt{y -\frac{w}{\lambda}}$ and the last term becomes
$$\int_{-\infty}^{\lambda y} \Bigl(H_n\Bigl(y - \frac{w}{\lambda}\Bigr) - \frac{A_n}{\sqrt{\pi}}\Bigr) e^{-\ep(y - (w/\lambda)}\sqrt{y} \Bigr(\int_{-2}^{2} \Bigr( 1 - \frac{|u|}{2}\Bigr) e^{\i w u} du\Bigr) \frac{ dw}{\sqrt{y - (w/\lambda)}}$$
$$ =2\int_{-\infty}^{\lambda y} H_n\Bigl(y - \frac{w}{\lambda}\Bigr)e^{-\ep(y - w/\lambda)} \frac{\sin^2 w}{w^2} \frac{\sqrt{y}}{\sqrt{y - w/\lambda}} dw $$
$$- \frac{2 A_n}{\sqrt{\pi}}\int_{-\infty}^{\lambda y} \frac{\sin^2 w}{w^2}\frac{\sqrt{y}}{\sqrt{y - w/\lambda}}e^{-\ep(y - w/\lambda)} dw.$$
Now, as in \cite{KS}, we take the limit $\ep \searrow 0$ and set $\kc_{0, n}(t) = \lim_{\ep \searrow 0} \kc_{\ep, n}(t)$. By the Lebesgue convergence theorem and  Sub-Lemma 4.5 in \cite{KS} we obtain
\begin{eqnarray} \label{eq:5.5}
\lim_{y \to \infty}\frac{1}{2}\int_{-2\lambda}^{2 \lambda} \kc_{0, n}(t) \Bigr( 1 - \frac{|t|}{2\lambda}\Bigr) \sqrt{y} e^{\i t y} dy + A_n \sqrt{\pi}\nonumber \\
= \lim_{y \to \infty}\int_{-\infty}^{\lambda y} H_n\Bigl(y - \frac{w}{\lambda}\Bigr) 
\frac{\sin^2 w}{w^2} \frac{\sqrt{y}}{\sqrt{y - w/\lambda}} dw.
\end{eqnarray}
By using an integration by parts and the fact that $\kc_{0, n}(t) \in W^{1, 1}_{loc}(\R)$, we may deduce that for every fixed $\lambda > 1$ the first term on the left hand side of (\ref{eq:5.5}) has a limit 0 as $y \to \infty.$ 
However, for every fixed $0 < \eta < 1$ and fixed $\lambda > 1$ if we wish to arrange the inequality
\begin{equation}\label{eq:5.6}
\Bigr|\frac{1}{2}\int_{-2\lambda}^{2 \lambda} \kc_{0, n}(t) \Bigr( 1 - \frac{|t|}{2\lambda}\Bigr) \sqrt{y} e^{\i t y} dt\Bigr| \leq A_n \eta,
\end{equation}
we must take $y \geq Y(\eta, \lambda, n)$ and we may have $Y(\eta, \lambda, n) \to \infty$ as $n \to \infty, \lambda \to \infty.$\\

By using the representation (\ref{eq:5.1}) and the estimates (\ref{eq:5.2}), (\ref{eq:5.3}), we will prove a more precise result.

\begin{lem} Let $q \geq  0$ be fixed and let $y \geq n- q,\:\lambda =\lambda_n = e^{\frac{1}{2}\mu_0 n}$.Then  for any $\eta > 0$ there exists $n_0(\eta) \in \N$ such that for all $n \geq n_0(\eta) + q$ 
we have $(\ref{eq:5.6}).$
\end{lem}

\begin{proof}  First we treat the integral
$$J_n(y) = \int_{- 2\lambda_n}^{2 \lambda_n} L_n(1 + \i t) \Bigr( 1 - \frac{|t|}{2\lambda_n}\Bigr) \sqrt{y} e^{\i t y} dt,$$
where $L_n(1 + \i t) = \lim_{\delta \searrow 0} L_{n}(1 + \delta + \i t)$. We write this integral as follows
$$J_n(y) = -\int_{\gamma_1} \i L_n(s) \Bigl(1 - \frac{s-1}{2\i \lambda_n}\Bigr) \sqrt{y} e^{(s-1)y} ds$$
$$ -\int_{\gamma_2}\i L_n(s) \Bigl(1 + \frac{s-1}{2\i \lambda_n}\Bigr) \sqrt{y} e^{(s-1)y} ds,$$
where 
$$\gamma_1 = \{ s \in \C: s = 1 + \i t, \: 0 \leq t \leq 2 \lambda_n\},\: \gamma_2 =\{s \in \C:\: s = 1+ \i t, \: -2\lambda_n \leq t \leq 0\}.$$
Since $L_n(s)$ has an analytic continuation for $1 - \mu_0 \leq \Re s \leq 1+ \delta_0$, we have the equality 
$$\int_{\gamma_1} + \int_{\omega_{1,1}} + \int_{\omega_{1,2}} + \int_{\omega_{1,3}} = 0, $$
where the function under integration is $-\i L_n(s) \Bigl(1 - \frac{s-1}{2\i \lambda_n}\Bigr) \sqrt{y} e^{(s-1)y}$ and
$$\omega_{1, 1} = \{ s \in \C:\: s = z + 2 \i \lambda_n,\: 1- \mu_0 \leq z \leq 1\},$$
$$\omega_{1, 2}= \{ s \in \C:\: s = 1 - \mu_0 + \i t,\: 0 \leq t \leq 2 \lambda_n\},\:\omega_{1, 3} = \{s \in \R:\: 1- \mu_0 \leq s \leq 1\}$$
with suitable orientation (see Figure 1). The integral over $\omega_{1, 2}$ has the form
$$-e^{-\mu_0 y} \sqrt{y} \int_0^{2\lambda_n} L_n(1- \mu_0 + \i t) \Bigl(1 + \frac{\mu_0 - \i t}{2\i \lambda_n}\Bigr)e^{\i t y} dt.$$

\begin{figure}[htpb]
\begin{center}
\begin{picture}(0,0)%
\includegraphics{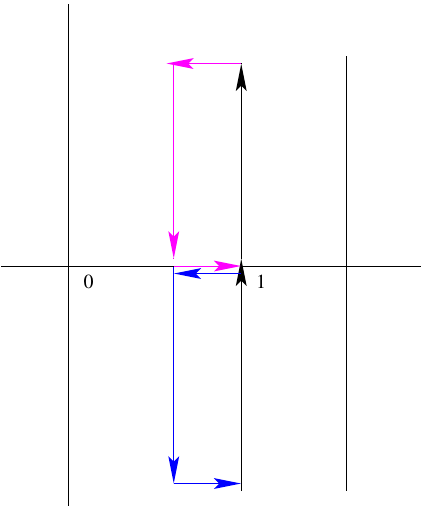}%
\end{picture}%
\setlength{\unitlength}{1895sp}%
\begingroup\makeatletter\ifx\SetFigFont\undefined%
\gdef\SetFigFont#1#2#3#4#5{%
  \reset@font\fontsize{#1}{#2pt}%
  \fontfamily{#3}\fontseries{#4}\fontshape{#5}%
  \selectfont}%
\fi\endgroup%
\begin{picture}(4224,5115)(4939,-5164)
\put(6226,-5086){\makebox(0,0)[lb]{\smash{{\SetFigFont{6}{7.2}{\familydefault}{\mddefault}{\updefault}{\color[rgb]{0,0,0}-2$\lambda_n$}%
}}}}
\put(7651,-1711){\makebox(0,0)[lb]{\smash{{\SetFigFont{6}{7.2}{\familydefault}{\mddefault}{\updefault}{\color[rgb]{0,0,0}$\gamma_1$}%
}}}}
\put(7576,-3811){\makebox(0,0)[lb]{\smash{{\SetFigFont{6}{7.2}{\familydefault}{\mddefault}{\updefault}{\color[rgb]{0,0,0}$\gamma_2$}%
}}}}
\put(6151,-2536){\makebox(0,0)[lb]{\smash{{\SetFigFont{6}{7.2}{\familydefault}{\mddefault}{\updefault}{\color[rgb]{0,0,0} 1-$\mu_0$}%
}}}}
\put(6226,-511){\makebox(0,0)[lb]{\smash{{\SetFigFont{6}{7.2}{\familydefault}{\mddefault}{\updefault}{\color[rgb]{0,0,0}2$\lambda_n$}%
}}}}
\end{picture}%

\caption{Contour of integration for $L_n(1 + \i t)$}
\end{center}
\end{figure}

We integrate by parts and one obtains
$$\frac{e^{-\mu_0 y} }{\i \sqrt{y}} \int_0^{2\lambda_n} \frac{d }{dt} \Bigl[L_n(1- \mu_0 + \i t) \Bigl(1 + \frac{\mu_0 - \i t}{2\i \lambda_n}\Bigr)\Bigr]e^{\i t y} dt$$
$$ - \frac{e^{-\mu_0 y} }{\i\sqrt{y}} \Bigl(\frac{\mu_0}{2 \i \lambda_n} L_n( 1- \mu_0 + 2 \i \lambda_n)e^{2 \i \lambda_n y} - \Bigl(1 + \frac{\mu_0}{2\i \lambda_n}\Bigr)L_n( 1 - \mu_0)\Bigr).$$

The function $L_n(1+ \i t)$ has an analytic continuation $L_n(z + \i t)$ for $1- \mu_0 \leq z \leq 1 + \delta_0$ and for any $0 < \nu < 1$ and any $n$ we have the estimate
$$\Bigl|\frac{d}{dt} L_n(z + \i t)\Bigr|\leq B(\nu) (1 + |t|^{\nu}),\: 1- \mu_0 \leq z \leq 1 + \delta_0.$$
Thus for $y \geq  n- q$ and large $n$ one gets
\begin{eqnarray*} \
\Bigl|\frac{e^{-\mu_0 y} }{\i \sqrt{y}} \int_0^{2\lambda_n} \frac{d }{dt} \Bigl[L_n(1- \mu_0 + \i t) \Bigl(1 + \frac{\mu_0 - \i t}{2\i \lambda_n}\Bigr)\Bigr]e^{\i t y} dt\Bigr|\nonumber \\
\leq C_{1,2}(\nu) \frac{e^{-\mu_0 y}}{\sqrt{y}}\lambda_n^{1 + \nu} 
\leq C_{1,2}(\nu) e^{-\mu_0 (n - q)} e^{\frac{1}{2}(1 + \nu) \mu_0 n}\nonumber \\ = C_{1,2}(\nu) e^{\mu_0 q}  e^{(-\frac{1}{2}+ \frac{1}{2} \nu) \mu_0 n}. 
\end{eqnarray*}
We choose $\nu =1/3$ and for the last term of the above inequality one obtains a bound ${\mathcal O}(e^{-\frac{1}{3}\mu_0 n}).$ Since $A_n \geq C_0 e^{-\frac{\mu_0}{4} n}$, we obtain a term $A_n o(n).$ The boundary terms are easily estimated and we get
\begin{equation}\label{eq:5.7}
|\int_{\omega_{1,2}}| \leq C_{1.2} e^{-\frac{1}{3} \mu_0 n}.
\end{equation}
Passing to the integral over $\omega_{1, 1}$, notice that for $ s \in \omega_{1,1}$ we have
$$1 - \Bigl(\frac{s-1}{2\i \lambda_n}\Bigr) = \frac{1 -z}{2 \i \lambda_n},\: |e^{(s-1) y}| \leq e^{(z-1)y} \leq 1, \: \Re s = z.$$
We integrate by parts with respect to $z$ and deduce
$$|\int_{\omega_{1,1}}| \leq A_{1,1}(\nu)\frac{ e^{-\mu_0 y}}{\sqrt{y}}\frac{(1 + \lambda_n^{\nu})}{\lambda_n}+  \frac{1}{2\sqrt{y}\lambda_n}\int_{1-\mu_0}^1 \Bigl|\frac{d}{dz}\Bigl[(1- z)L_n(z + 2\i \lambda_n)\Bigr]\Bigr| dz .$$  
Therefore, applying (\ref{eq:5.2}) for the second term in the right-hand-side, one obtains
\begin{equation} \label{eq:5.8}
|\int_{\omega_{1,1}}| \leq \frac{C_{1,1}(\nu)}{\sqrt{n}} e^{-\frac{1}{2}(1- \nu)\mu_0 n}  \leq C_{1, 1} e^{-\frac{1}{3}\mu_0n},
\end{equation}
choosing $\nu = 1/3.$
Before treating the integral over $\omega_{1, 3}$, consider the equality
$$\int_{\gamma_2} + \int_{\omega_{2,1}} + \int_{\omega_{2,2}} + \int_{\omega_{2,3}} = 0, $$
where the function under integration is $-\i L_n(s) \Bigl( 1 +\frac{s-1}{2\i \lambda_n}\Bigr) \sqrt{y}e^{(s-1)y}$ and
$$\omega_{2,1} = \{ s \in \C:\: s = z - 2 \i \lambda_n,\: 1- \mu_0 \leq z \leq 1\},$$
$$\omega_{2, 2}= \{ s \in \C:\: s = 1 - \mu_0 + \i t,\: -2 \lambda_n \leq t \leq 0\},\:\omega_{2, 3} = \{s \in \R:\: 1- \mu_0 \leq s \leq 1\}$$
with suitable orientation (see Figure 1). In particular, the curves $\omega_{1,3}$ and $\omega_{2,3}$ coincide, but they have inverse orientations.
The analysis of $\int_{\omega_{2,2}}$ is completely similar and one obtains (\ref{eq:5.7}). For $ s \in\omega_{2,1}$ we have
$$1 + \Bigl(\frac{s-1}{2\i \lambda_n}\Bigr) = \frac{z -1}{2 \i \lambda_n},\: |e^{(s-1) y}| \leq e^{(z-1)y} \leq 1$$
and as above one has (\ref{eq:5.8}). Now we take the sum of the integrals over $\omega_{1,3}$ and $\omega_{2,3}$ and we are going to estimate the integral 
$$\frac{1}{\lambda_n} \int_{1- \mu_0}^{1}  (z-1) \sqrt{y} L_n(z) e^{(z-1)y}dz.$$
We integrate by parts and the analysis is reduced to the integral
$$\frac{1}{\sqrt{y}\lambda_n} \int_{1- \mu_0}^1 e^{(z-1)y} \frac{d}{dz}\Bigl( (z-1)L_n(z)\Bigr) dz$$
which can be estimated by $C_{1,3}e^{-\frac{1}{2}\mu_0n}.$\\

Next we pass to the analysis of the integral
$$I_{n}(y) = \int_{- 2\lambda_n}^{2 \lambda} k_{n}(1 + \i t) \Bigr( 1 - \frac{|t|}{2\lambda_n}\Bigr) \sqrt{y} e^{\i t y} dt,$$
where 
$$k_{n}(1 + \i t) =  \lim_{\delta \searrow 0} K_{n}(1+\delta + \i t).$$
 Our purpose is to show that for $y \geq  n- q$ and $\lambda_n = e^{-\frac{1}{2} \mu_0 n}$ for any $0 < \eta \ll 1$ there exists $n_0(\eta)$ such that for $n \geq n_0(\eta) + q$ we have
$|I_{n}(y)| <\eta.$ We integrate by parts with respect to $t$ and deduce
$$I_{n}(y) = -\frac{1}{\i\sqrt{y}} \int_{-2\lambda_n}^{2\lambda_n} \frac{d}{dt} \Bigl[\Bigl(1 - \frac{|t|}{2\lambda_n}\Bigr)k_n(1 + \i t)\Bigr] e^{\i ty} dt $$
$$= \frac{\i}{\sqrt{y}}\Bigl( \int_{-M- 1}^{M +1}
+ \int_{-2 \lambda_n}^{-M- 1}+ \int_{M + 1}^{ 2 \lambda_n}\Bigr).$$
The integral over $[-M-1, M+ 1]$ can be estimated taking $n$ large by using the factor $\frac{1}{\sqrt{y}}$ and the fact that by hypothesis
$$\|k_{n}(1 + \i t)\|_{L^1(-M-1, M+1)}+\|k'_n(1 + \i t)\|_{L^1(-M - 1, M+ 1)} \leq C(M), \: \forall n \in \N.$$ 

\begin{figure}[htpb]
\begin{center}
\begin{picture}(0,0)%
\includegraphics{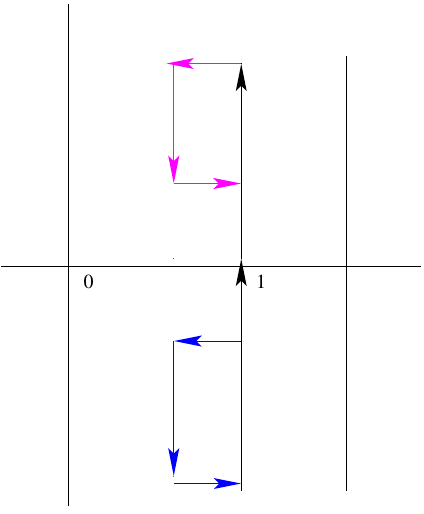}%
\end{picture}%
\setlength{\unitlength}{1895sp}%
\begingroup\makeatletter\ifx\SetFigFont\undefined%
\gdef\SetFigFont#1#2#3#4#5{%
  \reset@font\fontsize{#1}{#2pt}%
  \fontfamily{#3}\fontseries{#4}\fontshape{#5}%
  \selectfont}%
\fi\endgroup%
\begin{picture}(4224,5115)(4939,-5164)
\put(6376,-3436){\makebox(0,0)[lb]{\smash{{\SetFigFont{6}{7.2}{\familydefault}{\mddefault}{\updefault}{\color[rgb]{0,0,1}-M}%
}}}}
\put(7651,-1711){\makebox(0,0)[lb]{\smash{{\SetFigFont{6}{7.2}{\familydefault}{\mddefault}{\updefault}{\color[rgb]{0,0,0}$\gamma_1$}%
}}}}
\put(7576,-3811){\makebox(0,0)[lb]{\smash{{\SetFigFont{6}{7.2}{\familydefault}{\mddefault}{\updefault}{\color[rgb]{0,0,0}$\gamma_2$}%
}}}}
\put(6151,-2536){\makebox(0,0)[lb]{\smash{{\SetFigFont{6}{7.2}{\familydefault}{\mddefault}{\updefault}{\color[rgb]{0,0,0} 1-$\mu_0$}%
}}}}
\put(6226,-511){\makebox(0,0)[lb]{\smash{{\SetFigFont{6}{7.2}{\familydefault}{\mddefault}{\updefault}{\color[rgb]{0,0,0}2$\lambda_n$}%
}}}}
\put(6226,-5086){\makebox(0,0)[lb]{\smash{{\SetFigFont{6}{7.2}{\familydefault}{\mddefault}{\updefault}{\color[rgb]{0,0,0}-2$\lambda_n$}%
}}}}
\put(6451,-1786){\makebox(0,0)[lb]{\smash{{\SetFigFont{6}{7.2}{\familydefault}{\mddefault}{\updefault}{\color[rgb]{0,0,1}M}%
}}}}
\end{picture}%

\caption{Contour of integration for $K_n(1 + \i t)$}
\end{center}

\end{figure}

For the other two integrals we apply the argument used above exploiting the analytic continuation of $k_{n}(s)$ for $1 - \mu_0 \leq \Re s \leq 1, |\Im s| \geq M.$ We treat below only the integral over 
$[M+ 1, 2\lambda_n]$, the analysis of the other one  is very similar. We have
\begin{eqnarray*} 
\frac{1}{\sqrt{y}} \int_{M+1}^{2\lambda_n} \frac{d}{dt} \Bigl[\Bigl(1 - \frac{t}{2\lambda_n}\Bigr)k_{n}(1 + \i t)\Bigr] e^{\i ty} dt \nonumber \\
= -\frac{1}{2\lambda_n \sqrt{y}} \int_{M+1}^{2\lambda_n} k_{n}(1 + \i t) e^{\i t y} dt \nonumber \\+ \frac{1}{\sqrt{y}}\int_{M+1}^{2\lambda_n}\Bigl(1 - \frac{t}{2\lambda_n}\Bigr) \frac{d}{dt}k_{n}(1 + \i t) dt.
\end{eqnarray*}
We write the term on the right hand side as follows
$$\frac{\i}{2 \lambda_n \sqrt{y}} \int_{\beta_1} k_{n}(s)e^{(s-1)y} ds - \frac{1}{\sqrt{y}} \int_{\beta_1}\Bigl(1- \frac{s - 1}{2\i \lambda_n}\Bigr)\frac{d}{ds}(k_{n}(s)) ds,$$
where 
$$\beta_1 = \{s \in \C:\:s = 1 + \i t,\: M+1 \leq t \leq 2 \lambda_n\}.$$

The integral is equal to a sum of three integrals over the curves 
$$\beta_{1,1} = \{s \in \C:\: s = z + 2 \i \lambda_n,\: 1-\mu_0 \leq z \leq 1\},$$
$$\beta_{1, 2} =\{s \in \C:\: s = 1-\mu_0 + \i t,\: M+1 \leq t \leq 2 \lambda_n\},$$
$$ \beta_{1,3} = \{s \in \C:\:s = z + \i (M+1), \: 1-\mu_0 \leq z \leq 1\}$$ 
with suitable orientation (see Figure 2).
For the integral over $\beta_{1,2}$ we obtain an estimate ${\mathcal O}\Bigl(\frac{e^{-\mu_0 y}\lambda_n^{1 + \nu}}{\sqrt{y}} \Bigr) = {\mathcal O}\Bigl(e^{-\frac{1}{3}\mu_0 n}\Bigr)$, choosing $\nu = 1/3$, 
while for that over $\beta_{1, 1}$ one deduces an estimate ${\mathcal O}(\frac{\lambda_n^{-1 + \nu}}{\sqrt{y}})$ which yields the same bound. To investigate the integral over $\beta_{1.3},$ we use the factor $\frac{1}{\sqrt{y}}$ and the fact that $M$ is fixed. This completes the proof of Lemma 5.3.
\end{proof}

\begin{lem}For any $0 < \eta \ll 1$ there exists $n_0(\eta) \in \N$ such that for $y \geq 1,\:\lambda_n = e^{\frac{1}{2}\mu_0 n}$ we have
\begin{equation} \label{eq:5.10}
\Bigl|\int_{-\infty}^{\lambda_n y} \frac{\sqrt{y}}{\sqrt{y - w/\lambda_n}} \frac{\sin^2 w}{w^2}dw - \pi\Bigr| < \eta,\: n \geq n_0(\eta).
\end{equation}
\end{lem}

The proof is a repetition of the proof of Sub-Lemma 4.5 in \cite{KS} and we leave the details to reader.

Combining Lemmas 5.3 and 5.4, one obtains from the equality (\ref{eq:5.5}) that for fixed $q \geq 0$ and any $0 < \eta \ll 1$ there exists $n_0(\eta) \in \N$ such that for $y \geq n- q$ and $n \geq n_0(\eta) + q$ we have
\begin{equation} \label{eq:5.11}
A_n \sqrt{\pi}(1- \eta) \leq \int_{-\infty}^{\lambda_n y} H_n\Bigl(y - \frac{w}{\lambda_n}\Bigr) \frac{\sin^2 w}{w^2} \frac{\sqrt{y}}{\sqrt{y - w/\lambda_n}}dw  <A_n  \sqrt{\pi}(1 + \eta).
\end{equation}

Passing to the function $H_n(y)$, we have the following
\begin{lem} Let $q \geq 0$ be fixed. For any $0 < \eta \ll 1$ there exists $n_0(\eta) \in \N$ such that for $y \geq n- q$ and $n \geq n_0(\eta) + q$ we have
\begin{equation} \label{eq:5.12}
\frac{A_n}{\sqrt{\pi}}(1- \eta)\leq   H_n(y)  \leq \frac{A_n}{\sqrt{\pi}}(1 + \eta).
\end{equation}
\end{lem}
\begin{proof}
  Since $H_n(y)$ is nonnegative, from (\ref{eq:5.11}) for $y \geq n- q, \: n \geq n_0(\eta) + q$ it follows that
  \begin{equation} \label{eq:5.13}
\int_{-\sqrt{\lambda_n}}^{\sqrt{\lambda_n}} H_n\Bigl(y - \frac{w}{\lambda_n}\Bigr)  \frac{\sin^2 w}{w^2} \frac{\sqrt{y}}{\sqrt{y - w/\lambda_n}}dw \leq A_n\sqrt{\pi}(1 + \eta).
\end{equation}
By the monotonicity of $g_n(w)$ for $w \in [-\sqrt{\lambda_n}, \sqrt{\lambda_n}],$ we deduce
$$H_n\Bigl(y - \frac{w}{\lambda_n}\Bigr) \geq H_n\Bigl(y - \frac{1}{\sqrt{\lambda_n}}\Bigr)\exp\Bigl(- \frac{2}{\sqrt{\lambda_n}}\Bigr),$$
hence
$$H_n\Bigl(y - \frac{1}{\sqrt{\lambda_n}}\Bigr) \exp\Bigl(- \frac{2}{\sqrt{\lambda_n}}\Bigr)\int_{-\sqrt{\lambda_n}}^{\sqrt{\lambda_n}}\frac{\sin^2 w}{w^2} \frac{\sqrt{y}}{\sqrt{y - w/\lambda_n}}dw \leq A_n\sqrt{\pi}(1 + \eta)$$
and
$$H_n\Bigl(y - \frac{1}{\sqrt{\lambda_n}}\Bigr) \leq \frac{A_n\sqrt{\pi}(1 + \eta)e^{\frac{2}{\sqrt{\lambda_n}}}}{\int_{-\sqrt{\lambda_n}}^{\sqrt{\lambda_n}}\frac{\sin^2 w}{w^2} \frac{\sqrt{y}}{\sqrt{y - w/\lambda_n}}dw}.$$

As in Lemma 5.4, for large $n$ we can arrange
$$\Bigl|\int_{-\sqrt{\lambda_n}}^{\sqrt{\lambda_n}}\frac{\sin^2 w}{w^2} \frac{\sqrt{y}}{\sqrt{y - w/\lambda_n}}dw - \pi\Bigr| < \eta.$$

The above inequalities imply
$$H_n\Bigl(y - \frac{1}{\sqrt{\lambda_n}}\Bigr) \leq \frac{A_n}{\sqrt{\pi}}\Bigl[(1 + \eta) e^{\frac{2}{\sqrt{\lambda_n}}}\Bigr]\Bigl(1 - \frac{\eta}{\pi}\Bigr)^{-1}\leq \frac{A_n}{\sqrt{\pi}}(1 + C_3\eta)$$
with $C_3 >0$ independent on $n$. Replacing $y$ by $y + \frac{1}{\sqrt{\lambda_n}},$ we obtain for large $n$ the upper bound in (\ref{eq:5.12}). 

Next we pass to the analysis of the lower bound. Applying the upper bound obtained above and the monotonicity of $H_n(y)$, from the left-hand-side inequality in (\ref{eq:5.11}) we obtain
$$A_n\sqrt{\pi}(1 - \eta) \leq \frac{A_n}{\sqrt{\pi}}(1 + \eta)\int_{-\infty}^{-\sqrt{\lambda_n}}\frac{1}{w^2} dw $$
$$+\int_{-\sqrt{\lambda_n}}^{\sqrt{\lambda_n}}H_n\Bigl(y + \frac{1}{\sqrt{\lambda_n}}\Bigr) e^{\frac{2}{\sqrt{\lambda_n}}}\frac{\sin^2 w}{w^2}\frac{\sqrt{y}}{\sqrt{y - w/\lambda_n}}dw $$
$$+ \frac{A_n}{\sqrt{\pi}}(1 + \eta) \int_{\sqrt{\lambda_n}}^{\lambda_n y}\frac{\sin^2 w}{w^2}\frac{\sqrt{y}}{\sqrt{y - w/\lambda_n}}dw. $$

Clearly, for large $n$
$$\int_{-\infty}^{-\sqrt{\lambda_n}} \frac{1}{w^2} dw=  \frac{1}{\sqrt{\lambda_n}} <  \eta.$$
For the third term on the right hand side for large $n$ one gets
$$\int_{\sqrt{\lambda_n}}^{\lambda_n y} = \int_{\sqrt{\lambda_n}}^{\lambda_n y/2}+\int_{\lambda_n y/2}^{\lambda_n y} \leq \frac{1}{\sqrt{2}} \int_{\sqrt{\lambda_n}}^{\lambda_n y/2} \frac{\sin^2 w}{w^2} dw + \frac{4 \sqrt{y}}{\lambda_n^2 y^2} \int_{\lambda_n y/2}^{\lambda_n y}\frac{1}{\sqrt{y - w/\lambda_n}} dw$$
$$\leq \frac{1}{\sqrt{2}} \int_{\sqrt{\lambda_n}}^{\infty} \frac{\sin^2 w}{w^2} dw + \frac{4 \sqrt{2}}{\lambda_n y} < \eta.$$
Consequently,
$$A_n\sqrt{\pi}(1 -\eta)- 2\frac{A_n}{\sqrt{\pi}}(1+ \eta) \eta \leq H_n\Bigl(y + \frac{1}{\sqrt{\lambda_n}}\Bigr) e^{\frac{2}{\sqrt{\lambda_n}}}\int_{-\sqrt{\lambda_n}}^{\sqrt{\lambda_n}} \frac{\sin^2 w}{w^2}\frac{y}{\sqrt{y - w/\lambda_n}} dw$$
and for large $n$ with a constant $C_4 > 0$ independent on $n$ we obtain
$$\frac{A_n\sqrt{\pi}(1 - C_4 \eta)e^{-\frac{2}{\sqrt{\lambda_n}}}}{\pi + \eta}\leq \frac{A_n \sqrt{\pi}(1 - C_4 \eta)e^{-\frac{2}{\sqrt{\lambda_n}}}}{\int_{-\sqrt{\lambda_n}}^{\sqrt{\lambda_n}}\frac{\sin^2 w}{w^2}\frac{y}{\sqrt{y - w/\lambda_n}} dw} \leq H\Bigl(y + \frac{1}{\sqrt{\lambda_n}}\Bigr).$$
This estimate implies a lower bound for  $y \geq n + \frac{1}{\sqrt{\lambda_n}}- q.$ Changing $q$ by $ q+ 1$, we obtain a lower bound for $y \geq n - q.$
\end{proof}

Obviously Proposition 5.1 follows from Lemma 5.5.\\

For functions which are not monotonic we prove the following
\begin{prop}
   Let $g_n(t) \in C^1([0, \infty[; \R^+), \: n \in \N$, be nonnegative functions such that 
\begin{equation} \label{eq:5.14}
\max_{0 \leq t \leq 1}g_n(t) \leq B_0,\:
|g_n'(t)|\leq B_1 \frac{e^t}{\sqrt{t}},\:  t \geq 1,\: \forall n \in \N,
\end{equation}
with constants $B_0 > 0, B_1 > 0$ independent of $n$.
 Assume that for any $n \in \N$ the Laplace transforms
$$F_n(s) = \int_0^{\infty} e^{-s t} g_n(t)dt$$
are analytic for $\Re s > 1$ and have the same properties as in Proposition $5.1.$ 
Then for any fixed $q \geq 0$ and any $0 < \eta \ll 1$ there exists $n_0(\eta) \in \N$ such that for $t \geq n- q$ and $n \geq n_0(\eta)+ q$ we have
\begin{equation} \label{eq:5.15}
\frac{A_n e^t}{\sqrt{\pi t}} (1 -\eta) < g_n(t)  < \frac{A_n e^t}{\sqrt{\pi t}} (1+ \eta).
\end{equation}
\end{prop}
\begin{proof}  Replacing the function $g_n(t)$ by $\tilde{g}_n(t) = g_n(t) \sqrt{t}$, we can assume that
$$ F_n(s) = \int_0^{\infty} \frac{e^{-st}}{\sqrt{t}} \tilde{g}_n(t) dt$$
has the representation (\ref{eq:5.1}). On the other hand,
$$g_n(t) = \int_1^t g_n'(\sigma) d \sigma + g_n(1) \leq B_0 + B_1 \int_1^t \frac{e^{\sigma}}{\sqrt{\sigma}} d \sigma < B_0 + B_1e^t, \:t \geq 1,\: \forall n \in \N.$$
Therefore 
\begin{equation} \label{eq:5.16}
|\tilde{g}_n'(t)| = \Bigl|\frac{g_n(t)}{2 \sqrt{t}} + \sqrt{t} g_n'(t)\Bigr| \leq B_3 e^t,\:t \geq 1, \: \forall n \in \N
\end{equation}
with a constant $B_3 > 0$ independent on $n$.
 Below we denote $\tilde{g}_n(t)$ again by $g_n(t)$ and we assume that $|g_n'(t)| \leq B_3 e^t$ for $t \geq 1.$ We will repeat a part of the proof of Proposition 5.1 and for simplicity we use the notations of this proof. 
 Set $H_n(y) = g_n(y) e^{-y}$ and for $s = 1 + \ep + \i t$ define
$$\kc_{\ep, n}(t) =  F_n(s) - \frac{A_n}{\sqrt{s- 1}}= A_n K_n(s) + L_n(s).$$
Since for $F_n(s), K_n(s), L_n(s)$ we have the same assumptions as in Proposition 5.1, we can apply Lemma 5.3. 
Thus for $y \geq n - q,\: \lambda_n = e^{\frac{1}{2}\mu_0 n}$ and $n \geq n_0(\eta) + q$ we obtain the estimate (\ref{eq:5.11}).  

 Because $H_n(y) \geq 0$ to get an upper bound  we use the inequality
\begin{equation} \label{eq:5.17}
\int_{-\sqrt{\lambda_n}}^{\sqrt{\lambda_n}} H_n\Bigl(y - \frac{w}{\lambda_n}\Bigr) \frac{\sin^2 w}{w^2} \frac{\sqrt{y}}{\sqrt{y - w/\lambda_n}}dw  <  A_n  \sqrt{\pi}(1 + \eta).
\end{equation}
Clearly,  for $-\sqrt{\lambda_n} \leq w \leq \sqrt{\lambda_n}$ one has
$$H_n\Bigl(y - \frac{w}{\lambda_n}\Bigr) \geq g_n(y - \frac{w}{\lambda_n}) e^{-y- \frac{1}{\sqrt{\lambda_n}}}$$
$$= H_n(y) e^{-\frac{1}{\sqrt{\lambda_n}}} + \Bigl(g_n(y - \frac{w}{\lambda_n}) - g_n(y)\Bigr) e^{-y- \frac{1}{\sqrt{\lambda_n}}}.$$
By Taylor expansion write
$$g_n\Bigl(y - \frac{w}{\lambda_n}\Bigr) - g_n(y) = -\frac{w}{\lambda_n}g_n'\Bigl(y - \frac{\theta w}{\lambda_n}\Bigr),\: 0 < \theta < 1,$$
hence for $-\sqrt{\lambda_n} \leq w \leq \sqrt{\lambda_n}$ we deduce
$$|g_n\Bigl(y - \frac{w}{\lambda_n}\Bigr) - g_n(y)| \leq \frac{B_3}{\sqrt{\lambda_n}} e^{y + \frac{1}{\sqrt{\lambda_n}}},$$
where we have used the estimate (\ref{eq:5.16}). Thus we obtain
$$H_n\Bigl(y - \frac{w}{\lambda_n}\Bigr) \geq H_n(y) e^{-\frac{1}{\sqrt{\lambda_n}}}- \frac{B_3}{\sqrt{\lambda_n}}$$
and for $y \geq n- q$ and large $n$ we conclude that
$$H_n(y) \leq \frac{A_n \sqrt{\pi}(1 + \eta)e^{\frac{1}{\sqrt{\lambda_n}}}}{\int_{-\sqrt{\lambda_n}}^{\sqrt{\lambda_n}} \frac{\sin^2w}{w^2}\frac{\sqrt{y}}{\sqrt{y - \frac{w}{\lambda_n}}}dw} 
+ \frac{B_3}{\sqrt{\lambda_n}}e^{\frac{1}{\sqrt{\lambda_n}}}$$
$$\leq \frac{A_n \sqrt{\pi}(1+ \eta)^2}{\pi - \eta} + A_n \eta \leq \frac{A_n}{\sqrt{\pi}}(1 + C_5 \eta)$$
with $C_5 > 0$ independent of $n$. This implies the upper bound in (\ref{eq:5.15}).

Passing to the lower bound, we repeat the argument of Lemma 5.5 and we are going to find an upper bound for
$$\int_{-\sqrt{\lambda_n}}^{\sqrt{\lambda_n}}H_n\Bigl(y - \frac{w}{\sqrt{\lambda_n}}\Bigr)\frac{\sin^2 w}{w^2} \frac{\sqrt{y}}{\sqrt{y - w/\lambda_n}}dw. $$
As above,  by Taylor expansion for $-\sqrt{\lambda_n} \leq w \leq \sqrt{\lambda_n}$ and large $n$ one deduces
$$H_n\Bigl(y - \frac{w}{\sqrt{\lambda_n}}\Bigr) \leq g_n(y -\frac{w}{\sqrt{\lambda_n}})e^{-y + \frac{1}{\sqrt{\lambda_n}}}$$
$$\leq H_n(y) e^{\frac{1}{\sqrt{\lambda_n}}} + \frac{B_3}{\sqrt{\lambda_n}} e^{\frac{2}{\sqrt{\lambda_n}}} ,$$
and we obtain easily the lower bound in (\ref{eq:5.15}). This completes the proof.
\end{proof}

As a preparation for the proof of the estimate (\ref{eq:5.14}) we establish the following

\begin{lem} Let $\omega(t) \in C^1([0,\infty[)$ be a nonnegative function such that for a fixed $0 < \nu_0 \leq \mu_0/4$ we have
the estimate
\begin{equation} \label{eq:5.18}
|\omega'(t)| \leq D_1 \frac{e^{(1+\nu_0) t}}{\sqrt {t}},\: \forall t \geq t_0 > 1.
\end{equation}
Assume that the Laplace transform
$$\Omega(s) = \int_0^{\infty} e^{-st} \omega(t)dt$$
is analytic for $\Re s > 1$. Assume that there exist  $A> 0, \delta_0 > 0, M > 0$ such that $\Omega(s)$ has a representation
\begin{equation} \label{eq:5.19}
\Omega(s) =\frac{A}{\sqrt{s - 1}} + K(s),
\end{equation}
where  $K(s)$ is a function which is analytic for $\Re s > 1$ and $K(s)$ satisfies the same assumptions as $K_1(s)$ in Proposition $5.1$. 
Then there exists a constant $D_0 > 0$ such that
\begin{equation} \label{eq:5.20}
\omega(t) \leq D_0 \frac{e^t}{\sqrt{t}},\: \forall t \geq t_0.
\end{equation}
\end{lem}

\begin{proof} We follow the proof of Proposition 5.6. First we replace the function $\omega(t)$ by $\tilde{\omega}(t) = \sqrt{t} \omega(t)$ and for the Laplace transform
$$\int_0^{\infty} \frac{e^{-s t}}{\sqrt{t}}\tilde{\omega}(t)dt $$
we have the representation (\ref{eq:5.1}) with $K(s)$ instead of $K_{n}(s)$, $L_n(s) = 0$ and $|\tilde{\omega}'(t)|\leq C_2 e^{(1 + \nu_0)t}.$ We denote below $\tilde{\omega}(t)$ by $\omega(t).$ We set $H(y) = \frac{\omega(y)}{e^y}, \:\kc_0(t) = \lim_{\ep \searrow 0}K(1 +\ep + \i t)$ and prove the following
\begin{lem} There exist $y_0 > 1$ and $A_1 > 0$ independent on $y$ such that for $y \geq y_0, \lambda(y) = \exp(\frac{1}{2} \mu_0 y)$ we have
\begin{equation*}
\Bigr|\int_{-2\lambda(y)}^{2 \lambda(y)} \kc_{0}(t) \Bigr( 1 - \frac{|t|}{2\lambda(y)}\Bigr) \sqrt{y} e^{\i t y} dt\Bigr| \leq A_1.
\end{equation*}
\end{lem}
The proof is a repetition of that of Lemma 5.3 and we omit the details.

We return to the proof of Lemma 5.7. For $y \geq y_0$  and $\lambda(y) = \exp(\frac{1}{2}\mu_0 y)$ we deduce
$$ \int_{-\sqrt{\lambda(y)}}^{ \sqrt{\lambda(y)}}  H\Bigl(y - \frac{w}{\lambda(y)}\Bigr) \frac{\sin^2 w}{w^2} \frac{\sqrt{y}}{\sqrt{y - w/\lambda(y)}}dw  <  A  \sqrt{\pi} + A_1.$$
It is important to estimate for $-\sqrt{\lambda(y)} \leq w \leq \sqrt{\lambda(y)}$ and $0 < \theta < 1$ the term
$$\Bigl|\omega\Bigl(y - \frac{w}{\lambda(y)}\Bigr) - \omega(y)\Bigr| = \Bigl|\frac{w}{\lambda(y)} \omega'\Bigl(y - \frac{\theta w}{\lambda(y)}\Bigr)\Bigr| \leq 
D_1 e^{( y + \frac{1+ \nu_0}{\sqrt{\lambda(y)}})}e^{(\nu_0- \frac{1}{2}\mu_0)y}.$$
Since $0 < \nu_0 \leq \mu_0/4$, for large $y$ we may bound the right hand side of the last inequality by $c e^y$ with a small constant $c > 0$ independent on $y$ and $D_1$. Consequently, as in 
the proof of Proposition 5.6, one obtains
$$H\Bigl(y - \frac{w}{\lambda(y)}\Bigr)\geq H(y)e^{-\frac{1}{\sqrt{\lambda(y)}}} -B_3$$
with a constant $B_3 >0$ independent on $y$ and $D_1$ and we complete the proof as in Proposition 5.6.
\end{proof}
\begin{rem} Notice that our proof shows that the constant $D_0 > 0$ can be chosen independently of $D_1$ by taking $ t \geq t_1 > t_0$  and $t_1$ sufficiently large in $(\ref{eq:5.20}) $.

\end{rem}

\section{Asymptotic of $\rho_n(T)$}
\renewcommand{\theequation}{\arabic{section}.\arabic{equation}}
\setcounter{equation}{0}

In the section we use the notations of the previous sections. It is more convenient to study the function
$$g_n(T) = \ep_n e^{(-\gamma(a)  + 1)T}\rho_n(T) = \ep_n e^{(-\gamma(a) + 1)T}\int_{\rt} q_n(G^T - aT)(y) dm_F(y)$$
$$= \frac{\ep_n^2}{2\pi}e^{(-\gamma(a) + 1)T}\,  \int_{\R} \left( \int_{\rt}\, e^{z\, (G - a)^T(y)}\,d m_F(y)\right)\,  \hat{\chi}(\ep_n\omega)\,  \, d\omega.$$
Recall that  $\ep_n = e^{-\ep n},\:z = \xi(a) + \i \omega,\: a \in J  \Subset \Gamma_G,\: q_n(t) = e^{\xi(a) t} \chi(t/\ep_n).$ We take $0 <\ep \leq \mu_0/8$, $\mu_0$ being the constant in Proposition 4.2. We extend $g_n(T)$ as 0 for $T < 0$ and we wish to apply Proposition 5.6 for $g_n(T)$. We are going to check the assumptions of Proposition 5.6, where
$$A_n = \frac{C(a)\hat{\chi}(0)}{\sqrt{2\beta''(\xi(a))}} \ep_n^2,\: \forall n \in \N,$$
with $C(a) > 0$ determined below.
Clearly, $g_n(T)$ is a nonnegative function. The Laplace transform of $g_n(T)$ becomes
$$F_n(s) = \frac{\ep_n^2}{2\pi}\int_0^{\infty} e^{-(s + \gamma(a) -1)T} \,  \int_{\R} \left( \int_{\rt} \, e^{z\, (G - a)^T(y)}\,d m_F(y)\right)\,  \hat{\chi}(\ep_n\omega)\,  \, d\omega dT$$
and by Proposition 4.1 (i) for $\Re s >1$ the function $F_n(s)$ is analytic. Next we write the integral with respect to $\omega$ as a sum
$$\int_{\R} (...) = \int_{|\omega| \leq \ep_0}+\int_{\ep_0 < |\omega| \leq M} + \int_{|\omega| \geq M},$$
where $\ep_0 >0$ and $M >0$ are the constants introduced in the proof of Proposition 4.2. The corresponding decomposition of $F_n(s)$ will be a sum $F_n(s) = F_{1, n}(s) +F_{2, n}(s) + F_{3, n}(s).$ Notice that the factor $\hat{\chi}(\ep_n \omega)$ is not involved in the integration with respect to $y$ and $T$ and in the analysis of $F_{k,n}(s),\:k = 1,2$ we will have a coefficient $\ep_n^2$ implying the factor $A_n = {\mathcal O}(\ep_n^2).$  Moreover, in the Laplace transform $Z(s, \omega, a)$ we must replace  $s$ by $s +\gamma(a) - 1.$ According to Proposition 4.2,  the functions $F_{2, n}(s)$ are analytic for $1 - \mu_0 \leq\Re s, \: \mu_0 >0$ and 
$$\lim_{\delta \searrow 0}F_{2, n}(1 + \delta + \i t)= \ep_n^2f_{2, n}(1 + \i t)$$ 
with $f_{2, n}(1 + \i t)\in W^{1, 1}_{loc}(\R)$. For $|t| \leq M$ the function $|f_{2, n}(1 + \i t)|,\: |f_{2,n}'(1 + \i t)|$ are clearly uniformly bounded with respect to $n$, while for $|t| \geq M$ we have an analytic 
continuation $f_{2}(s)$ for $1 - \mu_0 \leq \Re s \leq 1+ \delta_0$. In the latter case we apply the estimate (\ref{eq:4.6}) with $0 < \nu < 1$ uniformly with respect to $n$ (see the case 2 in the proof of Proposition 4.2). Since 
$$\int_{\ep_0 \leq |\omega| \leq M} |\hat{\chi}(\ep_n \omega)| d\omega \leq  M \sup_{\xi \in\R}|\hat{\chi}(\xi)| ,\:\forall n \in \N,$$ 
 with a constant $C_2(\nu) > 0$ independent of $n,$ we obtain
\begin{equation} \label{eq:6.1}
|f_{2, n}(s)|\leq C_2(\nu)(1 + |\Im s|^{\nu}).
\end{equation}
Thus the term $F_{2, n}(s)$ contributes to $A_n K_n(s)$ in (\ref{eq:5.1}).\\

Passing to the analysis of $F_{3, n}(s)$, we apply the same argument based on the estimate (\ref{eq:4.6}). According to Proposition 4.2, $F_{3, n}(s)$ is analytic for $1 - \mu_0 \leq \Re s \leq 1 + \delta_0.$ In this case 
we have an infinite  integral with respect to $\omega$ and we will exploit the factor $\ep_n^2$ to estimate it. By using (\ref{eq:4.6}) with $0 < \nu < 1$, we must treat 
$$\ep_n^2 \int_{|\omega| \geq M} \Bigl(1 + |\Im s|^{\nu} + |\omega|^{\nu}\Bigr)|\hat{\chi}(\ep_n\omega)|d\omega.$$
The Fourier transform of $\chi \in C_0^{\infty}(\R)$ satisfies
$$|\hat{\chi}(\ep_n \omega)| \leq D_2(1 + |\ep_n \omega|)^{-2}$$
with a constant $D_2 > 0$ independent on $\ep_n$ and $\omega$.
By a change of variable $\ep_n \omega = \xi$ we get a convergent integral with respect to $\xi$ and with a constant $C_3(\nu) > 0$ independent on $n$ we deduce the estimate 
\begin{equation} \label{eq:6.2}
|F_{3, n}(t)|\leq C_3(\nu)\ep_n^{1-\nu}(1 +|\Im s|^{\nu}) \leq C_3(\nu) (1 + |\Im s|^{\nu}).
\end{equation}
Therefore, $F_{3, n}(s)$ contributes to the term $L_n(s)$ in (\ref{eq:5.1}) and we cannot get a coefficient $A_n$.\\

It remains to study the behaviour of $F_{1, n}(s)$. Here there are no problems with the convergence of the integral with respect to $\omega$ and we gain the factor $\ep_n^2.$ For $\ep_0 \leq |\Im s|$ the Laplace 
transform $F_{1,n}(s)$ has no singularities and as above we obtain  a contribution to $A_n K_n(s)$ in (\ref{eq:5.1}). Let $U_2 = \{s \in \C,\omega \in \R:\: |s - \gamma(a)| \leq \mu_0, |\omega| \leq \ep_0\}$. 
Recall that for $(s, \omega) \in U_2$  the function $Z(s, \omega, a)$ (independent on $n$) has a pole $s(\omega, a)$ and
$$Z(s + \gamma(a) -1, \omega, a) = \Bigl(\frac{B_3(s(\omega, a), \omega, a)}{\int \tau d\nu_{f_a -s(\omega, a) \tau + \i \omega g}}\Bigr) \frac {1}{s + \gamma(a) - 1- s(\omega, a)} + J_5(s, \omega, a)$$ 
with a function $J_5(s, \omega, a)$ analytic in $s$ and real analytic in $\omega.$ Here we may repeat without any change the argument in Section 5 in \cite{W}. Applying the Morse lemma to the function $\Re s(\omega, a)$, there exists a function $y =y(\omega, a)$ defined for $|\omega|\leq \ep_0$ such that
$$\Re s(\omega, a) = \gamma(a) - y^2.$$
Therefore the analysis is reduced to the integral
$$\frac{\ep_n^2}{2\pi}\int_{-\ep_0}^{\ep_0} \frac{B_3\Bigl(\gamma(a)- y^2(\omega, a) +\i q(\omega, a),\omega, a\Bigr)}{s - 1 + y^2(\omega, a) + \i q(\omega, a)} \hat{\chi}(\ep_n \omega)d\omega,$$
where (see Lemma 3 in \cite{W}) $q(\omega, a) = \Im s(\omega, a)$ is such that $q(0, a) =\frac{\partial q}{\partial \omega}(0, a)  = \frac{\partial^2 q}{\partial\omega^2}(0, a) = 0,$ and 
$$\frac{\partial ^2}{\partial \omega^2}\Re s(0 , a) = -\sigma^2(m_{F +\xi(a) G})= - \beta''(\xi(a)) < 0$$ 
with $\beta(\xi)$ and $\xi(a)$ introduced in Section 1.\\

Next the analysis follows  that in Section 3 in \cite{KS} without any change.
After a change of variable $ \omega = \omega(y,a)$ the integral has the representation 
$$\frac{\ep_n^2 C(a)}{ \pi}\frac{\hat{\chi}(0)}{\sqrt{2\beta''(\xi(a))}}\Bigl[\int_{-y(\ep_0,a)}^{y(\ep_0,a)}\frac{1}{s - 1 + y^2}dy$$ 
$$- \int_{-y(\ep_0,a)}^{y(\ep_0,a)}\frac{\i Q(y, a)}{(s - 1 + y^2 +\i Q(y, a))(s - 1 + y^2)}dy\Bigr]$$
$$+\frac{\ep_n^2}{2\pi} \frac{\sqrt{2}}{\sqrt{\beta''(\xi(a))}}\int_{-y(\ep_0, a)}^{y(\ep_0,a)}\frac{P(y)}{s - 1 + y^2 +\i Q(y, a)}dy,$$ 
where 
$$C(a) = \frac{1}{(\int \tau d\mu)^2}B_3(\gamma(a), 0, a),$$
 $P(y)$ is a complex valued function such that $P(0) = 0$ and $Q(y, a)$ is real valued odd function with respect to $y$ such that $Q(0, a) = Q'_y(0,a) = Q''_{yy}(0,a) = 0.$  Here $B_3(\gamma(a),0, a)$ is given by (\ref{eq:4.1}) and we have used that 
$$\frac{\partial y}{\partial \omega}(0,a) = \Bigl(\sqrt{-\frac{1}{2} \frac{\partial^2}{\partial \omega^2}\Re s(\omega, a)}\Bigr)\Bigl\vert_{\omega = 0} = \frac{\beta''(\xi(a))}{\sqrt{2}}$$
combined with the Taylor expansions for the functions $\Re s(\omega, a),\: \Im s(\omega, a)$ and $\hat{\chi}(\ep_n \omega)$ around $\omega = 0.$ In particular,
$$\hat{\chi}(\ep_n \omega) = \hat{\chi}(0)+ \ep_n \omega(y, a) \hat{\chi}'(0) + {\mathcal O}(\ep_n^2\omega^2(y, a)).$$
The first term in the above representation yields the singularity $\frac{A_n}{\sqrt{s-1}}$ plus more regular terms, where 
$$A_n = \frac{C(a) \hat{\chi}(0)}{\sqrt{ 2 \beta''(\xi(a)}}\ep_n^2.$$
The dependence on $n$ is caused by the coefficient $\ep_n^2$ involved in $A_n$.
Finally,  we obtain
$$F_{1, n}(s) = \frac{A_n}{\sqrt{s- 1}} + A_n \omega_{1, n}(s)$$ 
and for $|\Im s| \leq \ep_0$ we have uniform with respect to $n$ bounds for the $L^1(-\ep_0, \ep_0)$ norms of $f_{1, n}(1 + \i t)$ and $f_{1, n}'(1 + \i t)$. 
Summing the analysis of $F_{k, n}(s),\: k = 1,2,3$, we get the representation (\ref{eq:5.1}).\\

 Our purpose is to apply Proposition 5.6 for the functions $g_n(t)$. To do this, we need to estimate the derivative
\begin{eqnarray} \label{eq:6.3}
g_n'(t) = \ep_n (- \gamma(a) + 1 ) e^{(-\gamma(a) + 1)t} \int_{\rt} q_n(G^t- at)(y) dm_F(y) \nonumber \\
+ \ep_n  e^{(-\gamma(a) + 1)t} \xi(a)\int_{\rt}   q_n(G^t - at)(y)(G(\sigma_t^{\tau}(y)) - a) dm_F(y)\nonumber\\
+   e^{(-\gamma(a) + 1)t}\int_{\rt} e^{\xi(a)(G^t(y) - at)} \chi'\Bigl(\frac{G^t - a t}{\ep_n}\Bigr)(y)(G(\sigma_t^{\tau}(y) - a)dm_F(y).
\end{eqnarray}
We can choose a function $0 \leq \psi (t) \in C_0^{\infty}(\R)$ such that $\psi(t) = M_1 > 0$ for $t \in \supp\: \chi(\frac{t}{\ep_n}),$ where the constant $M_1$ (independent of $n$) is chosen so  that
$$ \max_{t \in \R}\chi(t) + \max_{t \in \R} |\chi'(t)| < M_1.$$
 Then for every fixed compact $J \Subset \Gamma_G$ and $a \in J$  there exists a constant $C(J) > 0$ independent of $n$ and $a$ such that
$$|g_n'(t)| \leq  C(J) e^{(-\gamma(a) + 1)t} \int_{\rt} e^{\xi(a)(G^t - at)}\psi(G^t - at) (y)dm_F(y)= C(J) \Psi(t).$$
Notice that $C(J)$ depends on $\xi(a), a$ and the maximum of $G(w)$, but $C(J)$ is independent of $M_1$.
The problem is reduced to an estimate of $\Psi(t).$ For the nonnegative function $\Psi(t)$ we wish to apply Lemma 5.7.  Consider the Laplace transform 
$$Y(s) = \int_0^{\infty} e^{-s t} \Psi(t) dt$$
for $\Re s > 1$ and its limit as $\Re s \searrow 1$. The analysis is completely the same as that of $F_n(s)$, where the function $\hat{\chi}_n(\omega)$ must be replaced by 
$\hat{\psi}(\omega)$ which is independent on $n$. Therefore with some constant $A > 0$ one deduces the representation 
$$Y(s) = \frac{A}{\sqrt{s-1}} +P(s)$$
with function $P(s)$ having the properties of $K(s)$ mentioned in Lemma 5.7. To satisfy the condition (\ref{eq:5.18}),
first as above we obtain an upper bound
\begin{equation} \label{eq:6.4}
|\Psi'(t)| \leq C_1(J) e^{(-\gamma(a) + 1)t} \int_{\rt} e^{\xi(a)(G^t - at)}\psi_1(G^t - at) (y)dm_F(y)= C_1(J) \Psi_1(t),
\end{equation}
where $0 \leq \psi_1(t) \in C_0^{\infty}(\R)$ is such that for $t \in \supp\: \psi(t)$ we have
$$\psi_1(t) \geq  \max_{t \in \R}\psi(t) + \max_{t \in \R} |\psi'(t)|.$$
We repeat this procedure once more and with a function $0 \leq \psi_2(t) \in C_0^{\infty}(\R)$ and constant $C_2(J)$ one arranges the bound
$$|\Psi_1'(t)| \leq C_2(J) e^{(-\gamma(a) + 1)t} \int_{\rt} e^{\xi(a)(G^t - at)}\psi_2(G^t - at)(y) dm_F(y)= C_2(J) \Psi_2(t).$$
Now the analysis in Section 4 shows that the Laplace transforms
$$\int_0^{\infty} e^{-st}\Psi_k(t)dt,\: k = 1, 2 , $$
are analytic for $\Re s > 1$. Therefore the integral
$$\int_0^{\infty} e^{-s t}\Psi_1'(t) dt$$
is absolutely convergent for $\Re s > 1$, and for $ s = 1 + \delta > 1$ we have
$$\int_0^{\infty} e^{-st}\Psi_1'(t) dt = \Bigl[e^{-st} \Psi_1(t)\Bigr]_{0}^{\infty} + s \int_0^{\infty} e^{-st} \Psi_1(t) dt ,$$
hence for every $\delta >0$ we have $\lim_{t \to \infty} e^{-(1 + \delta) t} \Psi_1(t) = 0.$
This estimate combined with (\ref{eq:6.4}) implies a bound 
$$|\Psi'(t)| \leq C_{\delta}C_1(J)\frac{e^{(1+ \delta)t}}{\sqrt{t}},\: t \geq 1$$
 and we are in position to apply Lemma 5.7 for the function $\Psi(t)$. The statement of Lemma 5.7 yields the estimate (\ref{eq:5.20}) for $\Psi(t)$ with a constant $D_0 >0$ independent 
 on $n,\:C_{\delta},\:C_1(J),\:C_2(J)$ and we obtain
\begin{equation} \label{eq:6.5}
|g_n'(t)|\leq D_0 C(J)\frac{e^t}{\sqrt{t}},\: \forall t \geq t_0,\: \forall n \in \N.
\end{equation}
Following Remark 5.9, we may take $t_0 > 1$ large to guarantee the independence of $D_0.$

On the other hand, it is clear that we may estimate $\max_{0 \leq t \leq 1}g_n(t)$ uniformly with respect to $n$ and $a \in J.$
Thus we can apply Proposition 5.6 for $g_n(t)$.  We cancel the coefficients $\ep_n$ and $e^t$ in the estimates for $g_n(t)$ and one concludes that for fixed $ q \geq 0,\:T \geq n- q$ and any $0 < \eta \ll 1$ 
there exists $n_0(\eta) \in \N$ such that for $n \geq n_0(\eta) + q$ we have
\begin{equation} \label{eq:6.6}
\frac{\ep_n C(a) \int \chi(t) dt}{\sqrt{2 \pi T \beta''(\xi(a)}}e^{\gamma(a) T}(1- \eta) \leq \rho_n(T) \leq \frac{\ep_n C(a) \int \chi(t) dt}{\sqrt{2 \pi T \beta''(\xi(a)}}e^{\gamma(a)T}(1+ \eta).
\end{equation}
 It is easy to see that if one examines the function
$$\tilde{\rho}_n(T) = \int_U \chi\Bigl(\frac{(G^T - aT)(y)}{\ep_n}\Bigr) dm_F(y),$$
then for fixed $ q \geq 0$ and any $0 < \eta \ll 1$ there exists $n_0(y) \in \N$ such that for $n \geq n_0(\eta) + q$ we get 
\begin{equation} \label{eq:6.7}
\frac{\ep_n C(a) \int \chi(t) dt}{\sqrt{2 \pi T \beta''(\xi(a)}}e^{\gamma(a) T}(1- \eta) \leq \tilde{\rho}_n(T) \leq \frac{\ep_n C(a) \int \chi(t) dt}{\sqrt{2 \pi T \beta''(\xi(a)}}e^{\gamma(a)T}(1+ \eta).
\end{equation}
since the Fourier transform $\ep_n\hat{\chi}(\ep_n \omega)$ must be replaced by the Fourier transform $\ep_n\hat{\chi}(\ep_n(\omega -\i \xi(a))$ and
$$\int e^{-\ep_n\xi(a) t} \chi(t) dt = \int\chi(t) dt + {\mathcal O}(\ep_n).$$
Approximating the characteristic function ${\bf 1}_{[-1, 1]}(t)$ by cut-off functions $\chi_{\pm}(t)\in C_0^{\infty}(\R; [0, 1])$ such that
$$\chi_{-}(t) \leq{\bf 1}_{[-1, 1]}(t) \leq \chi_{+}(t),$$
 we obtain for $T \geq n - q,\: n \geq n_0(\eta) + q$ the estimates (\ref{eq:1.1}) and this proves Theorem 1.3. 
 
 \medskip

 It is worth noting that the derivatives $\chi_{+}'(t)$ may increase for $-1 -\delta \leq t \leq -1$ and $1 \leq t \leq 1 + \delta$, but this is not important for the estimate (\ref{eq:6.5}) since we may arrange $D_0$ to be 
 independent of the derivatives $\chi'_{+}(t)$ choosing $t \geq t_1$. This reflects in the choice of $n_0(\eta)$ in (\ref{eq:6.5}). The same observation holds for the derivative $\chi'_{-}(t)$ in $- 1 \leq t \leq 1 - \delta$ 
 and $1- \delta \leq t \leq 1.$\\
 
Now it is easy to pass to the analysis of the intervals $\Bigl(-\frac{e^{-\ep T}}{T}, \frac{e^{-\ep T}}{T}\Bigr).$ In fact, let $q = 1,\:0 < \eta \ll 1$ and $n_0(\eta) \in \N$ be as above. Let $T \geq n_0(\eta) +1$ and let 
$N(\eta) \geq n_0(\eta) + 1$ be chosen so that   $N(\eta) \leq T \leq N(\eta) + 1$. Obviously, we have
$\Bigl(-e^{-\ep (N(\eta) +1)}, e^{-\ep (N(\eta) +1)}\Bigr) \subset \Bigl(-e^{-\ep T}, e^{-\ep T}\Bigr) \subset \Bigl(-e^{-\ep N(\eta)}, e^{-\ep N(\eta)}\Bigr).$
Now we may examine
$$\zeta(T; a) = m_F\Bigl\{ w \in \rt: \int_0^T G(\sigma^{\tau}_t (w)dt - a T\in \Bigl(-e^{-\ep T}, e^{-\ep T}\Bigr) \Bigr\}.$$
For $T \geq n_0(\eta) +1$ we deduce the estimates
\begin{equation} \label{eq:6.9}
\frac{2e^{-\ep T} e^{-\ep} C(a)}{\sqrt{2 \pi T \beta''(\xi(a)}}e^{\gamma(a) T}(1- \eta)  \leq   \frac{2e^{-\ep(N(\eta) + 1)} C(a)}{\sqrt{2 \pi T \beta''(\xi(a)}}e^{\gamma(a) T}(1- \eta) \leq \zeta(T; a),
\end{equation}
\begin{equation} \label{eq:6.10}
  \zeta(T; a) \leq \frac{2e^{-\ep N(\eta)} C(a)}{\sqrt{2 \pi T \beta''(\xi(a)}}e^{\gamma(a)T}(1+ \eta) \leq \frac{2e^{\ep} e^{-\ep T} C(a)}{\sqrt{2 \pi T \beta''(\xi(a)}}e^{\gamma(a)T}(1+ \eta).
\end{equation}
To obtain (\ref{eq:6.9}), we exploit $T \geq [N(\eta) + 1] - 1$, while for (\ref{eq:6.10}) we use $T \geq N(\eta).$ These estimates prove the statement of Theorem 1.4.

 \renewcommand{\theequation}{A.\arabic{equation}}
  \setcounter{equation}{0}  
  \section*{Appendix}

\noindent
{\it Proof of Proposition} 2.2. In what follows we will denote global generic constants by $C > 0$ and $c > 0$.

\ms

(a) We will first show that $\tG$ is constant on stable leaves of $R^\tau$.
Let $\xi, \eta \in R^\tau$ be on the same stable leaf of $R^\tau$. Thus, $\xi = \pi(x,t)$, $\eta = \pi(y,t)$
for some $x,y \in R_i$ with $\piU(x) = \piU(y) = z\in U_i$ and some $t\in [0,\tau(x))$. Then $\tau(y) = \tau(x) = \tau(z)$.
Moreover, $s = \tau( \pp^{n+1}(x)) = \tau( \pp^{n+1}(y))$ and also $s = \tau(\pp^n(\piU(\pp(y)))$. Finally,
$\piU(\pp(x))= \piU(\pp(y))$. Using all these in the definition of $\tG$,  gives $\tG(x,t) = \tG(y,t)$.

To prove (2.6), write
\begin{eqnarray}
h(x) - h(\pp(x))
& = &  \sum_{n=0}^\infty \left[g(\pp^n(x)) - g(\pp^n(\piU(x))) \right] \nonumber\\
&    & \qquad  -  \sum_{n=0}^\infty \left[g(\pp^n(\pp(x))) - g(\pp^n(\piU(\pp(x)))) \right] \nonumber\\
& = & g(x) - g(\piU(x)) + \sum_{n=1}^\infty \left[g(\pp^n(x)) - g(\pp^n(\piU(x))) \right] \nonumber \\
&    & \qquad  -  \sum_{n=0}^\infty \left[g(\pp^n(\pp(x))) - g(\pp^n(\piU(\pp(x)))) \right] \\
& = & g(x) - g(\piU(x)) + \sum_{n= 0}^\infty \left[g(\pp^{n+1}(x)) - g(\pp^{n+1}(\piU(x))) \right] \nonumber\\
&    & \qquad  -  \sum_{n=0}^\infty \left[g (\pp^n(\pp(x))) - g (\pp^n(\piU(\pp(x)))) \right] \nonumber\\
& = & g(x) - \left[ g(\piU(x)) + \sum_{n= 0}^\infty \left(g(\pp^{n+1}(\piU(x)))  - g(\pp^n(\piU(\pp(x)))) \right) \right] \nonumber. 
\end{eqnarray}

Next, for every $x \in \hR$, using the change of variable $t \mapsto s = t\, \tau(\pp^{n+1}(x))/\tau (x)$ in some of the integrals below,  we get
\begin{eqnarray*}
\tg(x) 
& = & \int_0^{\tau(x)} \tG(x,t)\; dt  =  \int_0^{\tau(x)} G(\piU(x),t) \, dt\\
&    & + \sum_{n=0}^\infty [ \int_0^{\tau(x)} G(\pp^{n+1}(\piU(x))), t\,  \tau(\pp^{n+1}(x))/\tau(x)) \; dt \\
&    &  -  \int_0^{\tau(x)}  G(\pp^n(\piU(\pp(x))), t\, \tau(\pp^{n+1}(x))/ \tau(x))\, dt ] \\
& = & g(\piU(x)) + \sum_{n=0}^\infty [ \int_0^{\tau(\pp^{n+1}(x))} G(\pp^{n+1}(\piU(x))), s) \; ds \\
&    & -  \int_0^{\tau(\pp^{n+1}(x))}  G(\pp^n(\piU(\pp(x))), s)\, ds ] \\
& = & g(\piU(x)) + \sum_{n= 0}^\infty \left(g(\pp^{n+1}(\piU(x)))  - g(\pp^n(\piU(\pp(x)))) \right) .
\end{eqnarray*}
This and (A.1) imply $h(x) - h(\pp(x)) = g(x) - \tg(x)$, thus proving (2.6).

It remains to prove that $\tG$ and $h$ are $\beta$-H\"older on $\hrt$ and $\hR$ respectively, for some $\beta > 0$.

Let $x \neq y$ belong to some $R_i \cap \hR$ and let $0 \leq t < \tau(x)$ and $0 \leq t' < \tau(y)$. We may assume $\tau(x) \leq \tau(y)$.
We have to estimate $|\tG(x,t) - \tG(y,t')|$. We will first estimate $|\tG(x,t) - \tG(y,t)|$.

Set $\tx = \piU(x)$ and $\ty = \piU(y)$ and $s = t\, \tau(\pp^{n+1}(x))/\tau(x)$. Then $d(\tx,\ty) \leq C (d(x,y))^\alpha$ for some
global constant $C > 0$. 

Let $2m$ (or $2m+1$) be the maximal positive integer so that $\pp^j(x), \pp^j(y)$ belong to the same rectangle $R_{i_j}$ for $j = 0,1, \ldots, 2m-1$.
Then by (2.2), 
$$c \leq d(\pp^{2m}(x), \pp^{2m}(y)) \leq (c/c_0) \gamma_1^{2m} \, (d(x,y))^\alpha ,$$ 
so $d(x,y) \geq \frac{c}{\gamma_1^{2m/\alpha}}$, and therefore
\be
(d(x,y))^{\alpha \alpha'/2} \geq \frac{c}{\gamma_1^{\alpha' m}} = \frac{c}{(\rho \gamma)^m} > \frac{c}{\gamma^m} .
\ee

For any integer $n = 0,1, \ldots, m-1$, using (2.2) and the latter,  we get
\begin{eqnarray}
d(\pp^{n+1}(\tx), \pp^{n+1}(\ty)) 
& \leq & \frac{d(\pp^{2m}(\tx), \pp^{2m}(\ty))}{c_0 \gamma^{2m-n}} \leq \frac{C}{\gamma^{2m-n}}\nonumber \\
& \leq & \frac{C}{\gamma^{m-n}} \cdot \frac{1}{\gamma^m} \leq  \frac{C}{\gamma^{m-n}} \cdot (d(x,y))^{\alpha \alpha'/2} .
\end{eqnarray}

We will now estimate $|\tG(x,t) - \tG(y,t)|$. We have
\begin{eqnarray*}
|\tG(x,t) - \tG(y,t)|
& \leq & |G(\tx,t) - G(\ty,t)| 
      + \sum_{n=0}^{m-1} \left|G(\pp^{n+1}(\tx)), s) - G(\pp^{n+1}(\ty)), s)\right|\\
&       & +  \sum_{n=0}^{m-1} \left| G(\pp^n(\piU(\pp(x))), s) - G(\pp^n(\piU(\pp(x))), s) \right|\\
&       & +  \sum_{n=m}^\infty \left|G(\pp^{n+1}(\piU(x))), s) - G(\pp^n(\piU(\pp(x))), s)\right|\\
&       & +  \sum_{n=m}^\infty \left|G(\pp^{n+1}(\piU(y))), s) - G(\pp^n(\piU(\pp(y))), s)\right|\\
&  =   & I + II + III + IV + V.
\end{eqnarray*} 
Clearly, $I \leq |G|_\alpha (d(\tx, \ty))^\alpha \leq C\, |G|_\alpha \, (d(x,y))^{\alpha^2}$.

Next,
\begin{eqnarray*}
II
&    =  & \sum_{n=0}^{m-1} \left|G(\pp^{n+1}(\tx)), s) - G(\pp^{n+1}(\ty), s)\right|\\
& \leq & \sum_{n=0}^{m-1} |G|_\alpha \, \left(  \frac{C}{\gamma^{m-n}} \cdot (d(x,y))^{\alpha \alpha'/2}\right)^\alpha
\leq C\, |G|_\alpha \, (d(x,y))^\beta .
\end{eqnarray*}
Similarly, $III \leq C\, |G|_\alpha \, (d(x,y))^\beta$. 

Since $\pp(\piU(x))$ and $\piU(\pp(x))$ are on the same stable leaf of some rectangle, it follows from (2.3) that
\begin{eqnarray*}
 \left|G(\pp^{n+1}(\piU(x))), s) - G(\pp^n(\piU(\pp(x))), s)\right|\\
 \leq  |G|_\alpha\, (d(\pp^{n+1}(\piU(x))) , \pp^n(\piU(\pp(x)))^\alpha
 \leq C\, |G|_\alpha \frac{1}{\gamma^{\alpha n}} .
\end{eqnarray*}
This and (A.2) yield
\begin{eqnarray*}
IV \leq  C\, |G|_\alpha\, \sum_{n=m}^\infty \frac{1}{\gamma^{\alpha n}}  \leq C\, |G|_\alpha\, \frac{1}{\gamma^{\alpha m}}
\leq C\, |G|_\alpha\, (d(x,y))^{\alpha^2 \alpha'/2}  C\, |G|_\alpha\,  (d(x,y))^{\beta} .
\end{eqnarray*}
In a similar way we obtain $V \leq C\, |G|_\alpha\,  (d(x,y))^{\beta}$. Thus,
$$|\tG(x,t) - \tG(y,t)| \leq C\, |G|_\alpha\,  (d(x,y))^{\beta} .$$

An estimate of the form $|\tG(y,t) - \tG(y,t')| \leq C |G|_\alpha \, |t-t'|^\beta$ can be obtained rather easily, and we leave the
details to the reader. This proves that $\tG$ is $\beta$-H\"older and $|\tG|_\beta \leq |G|_\alpha$.

The proof that $h$ is also $\beta$-H\"older is very similar to the above, in fact it is easier. We leave the details to the reader.

\ms

(b) The proof that $H$ is  $\beta$-H\"older is very similar to the proof above that $\tG$ is $\beta$-H\"older. We leave the details to the reader.

To establish (2.9), replace $H(x,t)$ in the integral in (2.9) by the right-hand-side of (2.8) and use the change of variable $t \mapsto s = t\, \tau(\pp^{n}(x))/\tau (x)$.
What we obtain in this way is the right-hand-side of (2.5). This proves (2.9).

Finally, to establish (2.10), write:
\begin{eqnarray*}
&    & H(x,t) - H(\pp(x), t\, \tau(\pp(x))/\tau(x))\\
& = &   \sum_{n=0}^\infty \left[ G(\pp^n(x), t\, \tau(\pp^n(x))/\tau(x) ) - G(\pp^n(\piU(x)), t\, \tau(\pp^n(x))/\tau(x) )\right]\\
&    & -  \sum_{n=0}^\infty \left[ G(\pp^{n+1}(x), t\, \tau(\pp^{n+1}(x))/\tau(x) ) - G(\pp^n(\piU(\pp(x))), t\, \tau(\pp^{n+1}(x))/\tau(x) )\right]\\
& = & G(x,t) - G(\piU(x), t)\\
&    & +  \sum_{n=1}^\infty \left[ G(\pp^n(x), t\, \tau(\pp^n(x))/\tau(x) ) - G(\pp^n(\piU(x)), t\, \tau(\pp^n(x))/\tau(x) )\right]\\
&    & -  \sum_{n=0}^\infty \left[ G(\pp^{n+1}(x), t\, \tau(\pp^{n+1}(x))/\tau(x) ) - G(\pp^n(\piU(\pp(x))), t\, \tau(\pp^{n+1}(x))/\tau(x) )\right]\\
& = & G(x,t) - G(\piU(x), t)\\
&    & +  \sum_{n=0}^\infty \left[ G(\pp^{n+1}(x), t\, \tau(\pp^{n+1}(x))/\tau(x) ) - G(\pp^{n+1}(\piU(x)), t\, \tau(\pp^{n+1}(x))/\tau(x) )\right]\\
&    & -  \sum_{n=0}^\infty \left[ G(\pp^{n+1}(x), t\, \tau(\pp^{n+1}(x))/\tau(x) ) - G(\pp^n(\piU(\pp(x))), t\, \tau(\pp^{n+1}(x))/\tau(x) )\right]\\
& = & G(x,t) - ( \; G(\piU(x), t) +  \sum_{n=0}^\infty [ G(\pp^{n+1}(\piU(x)), t\, \tau(\pp^{n+1}(x))/\tau(x) )\\
&    & - G(\pp^n(\piU(\pp(x))), t\, \tau(\pp^{n+1}(x))/\tau(x) )] \; )\\
& = & G(x,t) - \tG(x,t) .
\end{eqnarray*}
This proves (2.10).

\subsection*{Acknowledgment}
Thanks are due to the referee for his useful comments and suggestions.

 \end{document}